\documentclass[12pt,a4paper]{article}
\usepackage[utf8]{inputenc}
\usepackage[pdftex,colorlinks=true,linkcolor=blue,citecolor=blue]{hyperref}
\usepackage{float}
\usepackage{libertine}
\usepackage[libertine]{newtxmath}
\usepackage[hmargin={1.25in,1.25in},vmargin={1.25in,1.25in},nohead]{geometry}

\usepackage{amsthm}
\usepackage{mathtools,array}
\usepackage{soul}
\usepackage{natbib}
\usepackage[small,bf]{titlesec}
\titleformat*{\paragraph}{\bfseries}
\titlespacing{\paragraph}{0pt}{*1}{*1}
\usepackage[normalem]{ulem}
\usepackage{tabularray}
\usepackage{csquotes}

\titlelabel{\thetitle.\quad}
\usepackage{enumitem}
\setlist{noitemsep,nolistsep}
\setlist[enumerate]{label=\textit{\alph*})}
\setlist[itemize]{leftmargin = 1pc}
\makeatletter
\g@addto@macro\bfseries{\boldmath}
\makeatother
\usepackage[export]{adjustbox}
\usepackage{graphicx}
\usepackage{subcaption}
\usepackage[dvipsnames]{xcolor}
\long\def\ignore#1{}

\newtheorem{proposition}{Proposition}

\newtheorem{theorem}{Theorem}
\newtheorem{lemma}{Lemma}

\theoremstyle{definition}
\newtheorem*{definition}{Definition}

\colorlet{maxcolor}{yellow!30}
\colorlet{avicolor}{blue!10}
\colorlet{eiloncolor}{blue!20}
\newcommand{\MAX}[1]{{\sethlcolor{maxcolor}\hl{#1}}}
\newcommand{\AVI}[1]{{\sethlcolor{avicolor}\hl{#1}}}

\DeclarePairedDelimiter\splitted{\langle}{\rangle}
\mathtoolsset{centercolon,mathic}

\let\le\leqslant
\let\ge\geqslant
\let\leq\leqslant
\let\geq\geqslant

\def\a{\alpha}
\def\b{\beta}
\def\ep{\varepsilon}
\def\U{\mathrm{U}}
\newcommand{\N}[1]{N^\mathrm{#1\kern-1.5pt}}

\newcommand{\calJ}{{\cal J}}

\newcommand{\rednote}[1]{\textcolor{red}{#1}}

\begin{document}
\title{Bayesian Dissuasion with Bandit Exploration%
\thanks{Lehrer acknowledges the support of grants ISF 591/21 and DFG
KA 5609/1-1.
Solan acknowledges the support of the ISF grant 211/22.}}
\author{Massimo D'Antoni\footnote{Department of Economics and Statistics, University of Siena, Siena, Italy, E-mail: dantoni@unisi.it.} \and
Ehud Lehrer\footnote{Department of Economics, Durham University, Durham, DH13LB, UK,  E-mail: ehud.m.lehrer@durham.ac.uk.} \and
Eilon Solan\footnote{School of Mathematical Sciences, Tel-Aviv University, Tel-Aviv, Israel, 6997800, E-mail: eilons@tauex.tau.ac.il.} \and
and Avraham Tabbach\footnote{Faculty of Law, Tel-Aviv University, Tel-Aviv, Israel, 6997800, E-mail: adtabbac@tauex.tau.ac.il.}}
\date{\today}

\begin{titlepage}

\maketitle
\thispagestyle{empty}

\begin{abstract}
  \noindent
  We investigate a two-period Bayesian persuasion game, where the receiver faces a decision, akin to a one-armed bandit problem: to undertake an action, gaining noisy information and a corresponding  positive or negative payoff, or to refrain. The sender's objective is to dissuade the receiver from taking action by furnishing information about the payoff.
  Our findings describe the optimal strategy for the amount and timing of information disclosure. 
  In scenarios where the sender possesses knowledge of the receiver's first-period action or observes a noisy public signal correlated with it, the optimal strategy entails revealing information in the second period. If this alone proves to be insufficient to dissuade the receiver from acting, 
  supplementary information is provided in the first period. 
  In scenarios where information must be provided without conditioning on the receiver's first-period action, 
  the optimal strategy entails revealing information exclusively in the first period. 
  
\medskip
\begin{description}
\item[Keywords:] Persuasion games, information disclosure,  behavior-based strategy, one-armed bandit, explore-exploit, law enforcement, tax evasion, enforcement communication.
\item[JEL classification:] D31; D62; H23; K14; K42
\end{description}
\end{abstract}

\end{titlepage}

\section{Introduction}

In the Bayesian persuasion model, an informed sender communicates with an uninformed receiver to influence the receiver's choices. The key question is how much information should the sender disclose to maximize her payoff. In a multi-period game, timing is also crucial -- the sender needs to determine when to reveal the information. 

The classical model of \citet{KG} and its subsequent variations assumed that the sender is the only source of information for the receiver,
see the survey \citet{kamenica2019bayesian}. 
Recent papers have explored models where the receiver has an additional source of information, independent of the play.

In certain dynamic interactions, the receiver's choice affects not only the payoffs, but also the amount of information he acquires about the state of nature. For instance, in bandit problems, an agent who pulls a risky arm receives not only a stochastic payoff, but also information about the machine's probability of success, leading to an explore-exploit trade-off. This paper seeks to understand how this characteristic affects the timing and amount of information the sender should reveal to the receiver if her goal is to dissuade the receiver from taking certain actions.

In our model, the receiver can \emph{act} or \emph{refrain}. 
When the receiver acts, he obtains a gain or incurs a loss, depending on the state of nature. The state in which acting implies a gain is denoted by $G$ and the state under which it implies a loss is denoted by $L$.  If the receiver refrains, he receives neither a payoff nor information.  
If the receiver acts, a noisy signal is generated, whose distribution depends on the state of nature. This signal plays two roles: it determines the receiver's payoff and it provides information to the receiver about the true state of nature. The sender suffers a cost whenever the receiver acts. Hence, the sender's objective is to minimize the number of times the receiver acts.

An important motivation for our study is law enforcement, and, in particular,  tax evasion.\footnote{Our study applies naturally to any one-armed bandit problem where there is an additional player who has superior knowledge about the state of nature and who aims at dissuading the decision maker from acting, see examples in Section~\ref{sec:applications}.}
The sender is an enforcement agency and the receiver is an offender. The state of nature represents the amount of resources deployed for enforcement (or enforcement efficiency), which determines the probability of punishment. 
The state $G$, where the offender gains, corresponds to few resources (low efficiency) and state $L$, where the offender loses, corresponds to many resources (high efficiency).  
The offender has prior beliefs about the state of nature, and he can learn more about it by offending (i.e., selecting act) and observing whether he is detected or not (i.e., the signal). 
In other words, acting also serves as an exploration mechanism. The law enforcement agency, on its part, contemplates whether to reveal information about its resources (efficiency) to deter wrongdoing, or whether to remain silent leaving the offender in the dark. 

We investigate a two-period model.  The receiver has the incentive to act if he assigns a high probability to state $G$ and to refrain if he assigns a low probability to this state.  In the intermediate range of beliefs, it might still be beneficial for the receiver to act in period 1 despite the associated cost, because the information learned by this action in period 1 can generate a gain in period 2.

Since exploration hurts the sender, she may attempt to discourage it by revealing information in period 1 or by offering to reveal information in period 2.\footnote{In both cases, we mean revealing information before the receiver decides whether to act or refrain.}  Revealing information in period 2 may or may not depend on the receiver's action in period 1, or more generally on the information the sender possesses about the receiver's action in period 1. 

We investigate three distinct information provision structures that differ in information design and equilibrium outcomes: 
\begin{itemize}
\item \emph{Unconditional Information Provision}: the sender observes neither the receiver's actions nor his signals. In this scenario, any information revealed in period 2 is disclosed regardless of the action taken in, or on the outcome of, period 1. 
\item  \emph{Action-Based Information Provision}:
the sender observes the receiver's action in period 1 and can condition the revelation of information in period 2 on these actions. 
\item \emph{Signal-Based Information Provision}:
the sender observes the signal generated by the receiver's action, but not necessarily the action itself, and can condition the revelation of information in period 2 on this signal.  
\end{itemize}
Except when the prior belief of the state $G$ is below a certain threshold level and it is optimal for the sender to disclose no information, the optimal information structure is different for the three models. 
\medskip

\paragraph{Efficiency.}
The aforementioned information structures are evaluated by comparing their outcomes to the benchmark case in which no information is provided.  The number of times the receiver acts in this case is an \textit{upper} bound of the number of times he acts in any equilibrium when the sender can provide information.  The receiver's payoff in this benchmark case is a \textit{lower} bound of his equilibrium payoff when the sender can provide information.  Normalizing the payoff in state $G$ (resp.,~$L$) to 1 (resp.,~$-c$ for $c > 0$), the receiver's payoff in the benchmark case is also a \textit{lower} bound on the number of times the receiver acts in any equilibrium, and hence it is also a \textit{lower} bound on the sender's loss in equilibrium.

We show that in the three scenarios of information provision we study, the sender can reduce, to her own benefit, the number of times the receiver acts relative to the benchmark case. In essence, this is because the receiver, without information, may take actions that are detrimental to him in some states of the nature.
The most effective way to do so is when information revelation can be based on the receiver's actions. In this variation, the sender can, in many cases, attain the lower bound of her loss.  Under Signal-Based Information Provision, as long as the receiver's action and his signal are sufficiently correlated, the equilibrium strategies and outcomes are no different than under Action-Based Information Provision, so in many cases the sender can still achieve the lower bound of her loss. On the other hand, under Unconditional Information Provision, the sender cannot achieve the lower bound of her loss, and, not surprisingly, this one is the worst for her among the three variations under consideration.

Intuitively, the aim of information revelation when the sender can base her strategy on actions or signals is two-fold. First, the sender wants to dissuade the receiver from acting in period 1. 
This is achieved by conditional information revelation, which acts as a `reward' or incentive for the receiver to refrain in period 1. Second, the sender seeks to manipulate the receiver's belief in a way that maximizes her own payoff, similar to other persuasion models. Conditioning information provision on behavior allows the sender, in some instances, to tailor the information quantity such that the receiver obtains no benefit compared to the no-information scenario. This enables the sender to achieve the desired lower-bound outcome. However, when information is delivered unconditionally (in either period 1 or 2), it enhances the receiver's potential payoff compared to the no-information case. This is because the receiver can combine the provided information with his own signals to gain an advantage. Consequently, the sender cannot achieve the lower bound in this scenario.

\paragraph{Timing.}

In the Unconditional Information Provision model, the optimal strategy for the sender is to send a single message in period 1. This message either induces the receiver's prior belief to the maximum level that leads him to refrain in both periods, or it instills full certainty that the state of nature is $G$, prompting action in both periods. Interestingly, promising to disclose information in period 2 as an incentive to deter the receiver from acting in period 1 is inefficient. 
This is because the Unconditional Information Provision cannot serve as a reward for either acting or refraining. 
However, promising to disclose information in period 2 can dissuade the receiver from acting in period 1 when acting in period 1 entails costs. To further explain the inefficiency of disclosing information as an incentive, consider a scenario where acting in period 1 entails costs but provides valuable information for period 2 that may still justify it. A naive approach might suggest deterring action in period 1 by disclosing equivalent information in period 2. However, this would not suffice. The unconditional nature of the disclosure lets the receiver combine the disclosed information with the information gained from acting in period 1, granting him an advantage. To dissuade the receiver from acting in period 1 would necessitate an excessive amount of promised information in period 2 compared to what could be attained by providing information solely in period 1. 

In contrast, under Action-Based and Signal-Based Information Provision, the optimal strategy is to postpone information revelation to period 2, and to condition it on the receiver's action or on the public signal in period 1. Thus, in these scenarios, the information serves as a reward for refraining in period 1. When the prior belief is moderate, this reward is sufficient to achieve its goal. When the receiver has optimistic beliefs about the true state of nature, the reward is insufficient to ensure no action in period 1. In such instances, supplementary information must be disclosed at the outset of period 1. This early provided information guarantees that the information promised in period 2 can serve as an effective carrot.




\paragraph{The ``clear-message'' strategy.}
A feature common to all scenarios is that, in equilibrium, the sender utilizes a ``clear-message'' strategy. In such a strategy, the sender reveals that the state of nature is $G$ with positive probability, and otherwise maintains a sufficiently high level of uncertainty about the true state of nature. 
To this, the receiver responds with a ``compliance'' strategy, where he takes the harmful action only when the sender reveals that the state of nature is $G$. 

The reason why efficient information revelation features these equilibrium strategies is as follows. In our setup, the sender can ``pay'' the receiver with information, which the receiver utilizes to optimize his actions and payoffs. 
The receiver obtains a maximal payoff, normalized to 1, when he acts if the state of nature is $G$. When he is not completely sure that the state is $G$, and still takes the harmful action, his payoff is less than the maximal possible payoff, but the sender suffers the same cost. 
This means that when the receiver is unsure that the state is $G$ and still takes the harmful action, the sender, by revealing information with positive probability, can keep the payoff to the receiver unchanged, while at the same time reducing the harm to herself. Thus, it is part of the equilibrium strategy that the receiver acts only when he is certain that the state of nature is $G$.\footnote{Compare with \citet{ely2020moving,solan2021dynamic,Zhao}.} 

An interesting consequence of this equilibrium property is that the receiver does not utilize the exploration opportunities offered by his harmful action.  In equilibrium, the receiver only takes the informative action when he already knows that the actual state is $G$, rendering his actions void of any informational gain. Throughout the game, the receiver solely relies on the information supplied by the sender. Simply put, in equilibrium, the sender effectively eradicates the receiver's ability to gain private knowledge through her strategically optimal information provision.

\paragraph{Our Contribution.}
This paper presents two significant contributions to the literature. Firstly, it introduces a Bayesian persuasion model where the receiver can take actions that influence his own information. This unique feature in our model is, to the best of our knowledge, missing in the literature. Previous dynamic models have explored scenarios where the receiver receives exogenous information or where his actions affect the beliefs of subsequent receivers. Our model is the first to examine a situation where the receiver's actions directly impact his own beliefs, and thereby, his subsequent behavior.

Our second contribution is in providing new insights regarding optimal information revelation. We show the importance of the timing of information disclosure as a function on what the information disclosure can or cannot be based on. When the sender observes the receiver's actions or the public signal, information is largely backloading. That is, the sender delays information provision as much as possible.  Moreover, when backloading is insufficient to deter acting in period 1, supplementary information must be revealed in period 1. To the best of our knowledge, this is the first time such a phenomenon has been observed in the literature. On the other hand, when the sender lacks visibility into the receiver's actions or the public signal, information disclosure is solely frontloading, with the sender revealing information at the earliest possible opportunity.

\paragraph{Related Literature.}

The literature on Bayesian persuasion studies how an informed player can influence the beliefs and decisions of another player by strategically revealing information to maximize her own payoffs.
Numerous studies analyze static models.\footnote{See, e.g., \cite{MM,arieli2019private,Orlov,%
  Liao,Montes,Ui:b,ui.22:wp}.} There is also a growing literature that focus on dynamic models. These models allow the sender to provide partial information along the play, potentially  conditional on the history of the game.\footnote{See, e.g.,\cite{RSV, kremer.mansour.ea.14,
  ganglmair2014conversation, horner2016selling, ely.17, renault2017optimal,
  halac2017contests, che2018recommender, honryo, ely2020moving, BRV, Su,
  lehrer,Zhao}.} In some of these studies 
the receiver can obtain additional information beyond what the sender provides, for instance, from independent sources. 
The classes of dynamic games studied in the literature vary considerably. This makes comparisons among papers a hard task. 
We focus on the papers most relevant to our paper, highlighting the key differences. While these studies are similar to ours in certain aspects, for example, in utilizing information as an incentive tool, our paper offers a distinct element: the receiver can actively gather information through his actions, which he can then utilize to his benefit. This distinctive element, giving strategic exploratory quality to actions, distinguishes our model from the existing literature. A second difference between our paper and the literature concerns the research question.
While the literature assumes that the sender observes the receiver's actions, and therefore can condition information disclosure on these actions,
we study how varying the monitoring technology of the sender affects the resulting equilibrium. 




\citet{Zhao} explores a repeated game scenario wherein the sender's objective is to maximize the (discounted) frequency with which the receiver opts for the sender's preferred action. The optimal strategy for the sender involves revealing information precisely when the sender's preferred action does not align with the receiver's best response in the static game. In line with our findings, \citet{Zhao} demonstrates that by conditioning information disclosure on specific actions, the sender uses information as an incentive to encourage the receiver to favor the desired action. Notably, our study differs in that the receiver lacks observation of his own stage payoff and relies solely on information provided by the sender. It implies that in persuading the receiver to adopt a particular action, the sender should consider only the receiver's payoff, without factoring in additional information the receiver may acquire through his chosen action, as in our model. 

\citet{BRV} present a stopping model wherein a sender endeavors to persuade a receiver to accept an offer before a specified deadline. The receiver's decision to accept the offer hinges on the uncertainty of the prevailing state of the world. While the receiver has the option to wait and accumulate information, doing so incurs a cost attributed to the time value of money. Information, crucial to decision-making, can originate either from the sender, often denoted as inside information, or from an external source, referred to as outside information.
A key distinction between their model and ours lies in the nature of outside information, which, in their framework, is beyond the control of both the sender and receiver. In contrast, our model allows the receiver to acquire information not only from the sender, but also under his control, exhibiting a new aspect of the dynamics of information acquisition and strategic interaction. In \citet{BRV} the outside information is restricted to be decisive:  either revealing with some probability that the state is bad, or with some probability that the state is good. In our model, distinguished by its repeated game nature as opposed to a stopping scenario, the outside source of information is never decisive. The absence of conclusive outside information in our setting adds a layer of complexity to the interplay between beliefs, information acquisition, and strategic decisions.



\citet{ely2020moving} study the role of information as an incentive device in a dynamic moral hazard framework. An agent works on a task of uncertain difficulty. The principal knows the task difficulty and can provide information over time in order to induce the agent to work as long as possible. 
The main difference between this paper and ours is that in the former the sender is the only source of information, while in ours, the agent's action generates information as well. In addition, in our paper, the sender aims at dissuading the agent, while in their paper she aims at inducing him to work as long as possible. This distinction generates significant behavioral differences. Moreover, the model of \citet{ely2020moving} is a quitting game, while our model is a repeated game. Nevertheless, our paper shares similar features, in particular that information can function as an incentive device.

\cite{Orlov} study a continuous-time quitting game (e.g., exercising a real option), where the payoffs depend on two random parameters: the realization of a geometric Brownian motion which is publicly known, and a binary parameter known only to the sender. Two types of misalignment between the sender and the receiver are considered. In their model, the receiver obtain information about the binary parameter only through the sender. 

A recent paper by \citet{BestQuigley} shares with our paper the idea that an action of the receiver produces information that is used in the longer term. 
However, in this paper, there is no long-term interaction between the sender and the receiver, as the sender meets a sequence of short-run receivers, and questions such as the optimal timing of information disclosure to a receiver and the exploration/exploitation issues are not addressed.
Moreover, the focus of \citet{BestQuigley} is markedly different, as it aims to analyze how reputation can substitute for exogenous commitment and under what communication/monitoring technology the resulting equilibrium can replicate the outcome of (static) Bayesian persuasion, where exogenous commitment is instead assumed.  

\color{black}


\color{black}

Our leading example focuses on the interaction between law enforcement agencies and offenders. Several studies, including \cite{Hernandez} and the references therein, explore Bayesian persuasion of  this topic in static models. For instance, \cite{Lazear} demonstrates in one-shot interactions that providing information to potential offenders can increase compliance levels.
Similarly, \cite{LS} show that focusing law enforcement resources on a subgroup of offenders can increase deterrence. 
In this context, it is well established that future compliance is affected by punishment \citep[see, e.g.,][]{Dusek}.

\color{black}
\paragraph{The structure of the paper.} The paper is structured as follows. Section \ref{sec:set up} describes the model. 
Section \ref{sec:applications} discusses several applications, in particular, law enforcement and tax evasion. Section \ref{sec: no information} studies the benchmark case where the sender provides no information.
Section \ref{sec: uncoditional} examines the case where the sender provides information unconditionally.  Sections \ref{sec: action-based} and \ref{sec: signal-based} present the most intriguing results from a behavioral standpoint. In Section~\ref{sec: action-based}, the sender can base her messages to the receiver on the choice made by the receiver in period 1. In Section \ref{sec: signal-based}, the sender can base her messages on the signal generated by the receiver's action; in the law enforcement setup, the signal represents whether the offender was detected and punished for her offense or not.  All proofs are relegated to the appendix.
\section{The Model}
\label{sec:set up}

The game under consideration involves two players: a sender (``she") and a receiver (``he"). It is a two-period, incomplete-information game with two possible states of nature: $G$ and $L$.\footnote{$G$ stands for the state of nature where the receiver gains, $L$ for the state where he loses.} The prior probability of state $G$ is $q\in (0,1)$.

At the outset of the game, the sender is informed of the realized state, while the receiver is not. This incomplete information model sets the stage for strategic communication and persuasion, with the sender aiming to influence the receiver's decision making.

In each of the two periods, which we refer to as period 1 and period 2, the receiver can decide to `act' (A) or `refrain' (R). Acting is costly to the sender and it may imply either a gain or a loss to the receiver depending on the state of nature, while refraining implies zero payoff to both players.

Before the receiver decides whether to act or to refrain, the sender can send a message to the receiver from a set $M$ that contains at least two messages. Moreover, by acting, the receiver may obtain a (noisy) signal on the state of nature. Specifically, the signal is either positive ($+$) or negative ($-$). When the state is $G$ (resp., $L$), the probability of a positive signal is $\a_G$ (resp.,~$\a_L$), while the probability of a negative signal is $1-\a_G$ (resp.,~$1-\a_L$). We assume that $\a_G > \a_L$. That is, the probability of obtaining the positive signal is higher at the state $G$ than at the state $L$. In contrast, when choosing action R, the receiver obtains neither a stage payoff, nor an informative signal. Formally, we assume that, when choosing R, the signal is positive with probability one in both states.

The receiver's payoff when he acts depends on the signal realization. 
Namely, a positive signal implies a positive payoff, while a negative signal a negative payoff, hence the expected payoff is larger at $G$ than at $L$. We normalize payoffs so that the receiver's expected stage payoff is $1$ at $G$ and $-c$ at $L$, for some $c>0$.

Table \ref{table1} summarizes the state-dependent probability of signals and the expected payoffs of the receiver when taking actions A or R.

\begin{table}[ht]
  \centering
  \begin{tblr}{width=.9\linewidth,
    colspec={Q[c]|Q[1,c]Q[1,c]Q[1.75,c]|Q[1,c]Q[1,c]Q[1.75,c]|},
    row{3-5}={mode=math}
  }
& \SetCell[c=3]{c}Act (A) &&& \SetCell[c=3]{c}Refrain (R)\\
& \SetCell[c=2]{c}Signal probability&&  \SetCell[r=2]{c}Expected Payoff & \SetCell[c=2]{c}Signal probability&&\SetCell[r=2]{c}Expected Payoff \\
\text{State}& (+) & (-) & & (+) & (-) &\\
\hline\hline
G & \a_G & 1-\a_G & 1  & 1 & 0 & 0\\
\hline
L & \a_L & 1-\a_L & -c  & 1 & 0 & 0\\
\hline 
\end{tblr}
  \caption{Signals probabilities and the receiver's expected payoff}
  \label{table1}
\end{table}

The sender's stage payoff is determined entirely by the choice made by the receiver. If the receiver chooses A, the sender's stage payoff is $-1$; otherwise, it is $0$. The goal of each player is to maximize her or his sum of the two stage payoffs. In particular, the sender's goal is to minimize the number of times the receivers acts. 

The solution concept we employ is Stackelberg equilibrium:
the sender moves first and commits to a strategy (for the whole game),
and then the receiver reacts by choosing his strategy.

\section{Applications}
\label{sec:applications}

Our framework is applicable to the two-period one-armed bandit problem with a Bernoulli reward process, but with a twist. 
In addition to the decision maker, there is another agent, playing the role of the sender, who is aware of the true probability of success for the arm and can provide information about it to the decision maker, who plays the role of the receiver. The sender's goal is to discourage the decision maker from playing at all. This combination of Bayesian persuasion and bandit exploration applies naturally to various scenarios. 
We will discuss  four  examples that fit our model to varying degrees: law enforcement and tax evasion, international relations, parenting, and gambling.

\paragraph{Tax Enforcement.} 
Our leading example concerns law enforcement in general, and tax evasion in particular. A key feature of law enforcement is that the level of enforcement, as opposed to the level of punishment, is usually unknown to the potential offender. As a consequence, there is a policy question on the part of the police or the enforcement agency of whether to provide information to potential offenders. Law enforcers and policy makers can and sometimes do provide information regarding the enforcement techniques and levels. For example, in certain road segments there are signs indicating “speed camera enforced” or “high enforcement level”. Yet, in other cases, law enforcers and policy makers adopt a vagueness policy, keeping the level of enforcement efforts and techniques secret. One of the best examples is the top-secret computer software algorithm, known as the Discriminant Index Factor (DIF), used by the IRS for selecting tax returns for audits. This algorithm is said to be guarded like the “Coca Cola formula” \cite[][p 9]{harcourt}.

Imagine a repeated game between the tax administration (the sender) and a potential tax evader (the receiver). The tax evader should decide whether to evade taxes (A) or not (R) in each of the two periods, gaining some benefits and facing a penalty in case of detection. The probability of detection, which is a random variable determined by the tax administration's resource allocation or enforcement efficiency, can be either low, benefiting the evader, or high, leading to evader losses. This gives rise to a classical explore-exploit strategy for the tax evader. The tax administration, on its side, knows its own type and can strategically disclose information to or conceal it from the offender to minimize tax evasion.  

The tax administration can of course combat tax evasion by increasing the penalties for it or by improving its efficiency. However, in recent years, in order to overcome the limitations of enforcement based only on ex post sanctions, tax authorities have increasingly relied on strategies of ``enforcement communication''. The latter may include letters or emails with reminders about obligations and the risk of noncompliance, but also more specific information about the taxpayer position.\footnote{The role of enforcement communication as a strategy of compliance risk management is illustrated for example in \citet{EU2010}.} While this communication is often standardised (a reminder of the sanctions in case of evasion), in many cases it includes information about the individual taxpayer's level of compliance, based on data available to the tax administration. If the recipients are selected using data of the tax administration (elaborated through advanced statistical techniques, including machine learning techniques), the very reception of the message provides information on the ability to detect potential violations. This affects perception of the probability of an audit for the specific taxpayer.

These strategies have attracted the attention of behavioral economics, where aspects such as education to legality and civic responsibility, building of a trust relation, are considered. However, the optimal communication and sanction strategy can be analysed also in terms of standard rational behavior as optimal information transmission policy. 


Our model can improve the design of information provision strategies, and, in particular, the proper timing of information disclosure as a function of the administration's information about the taxpayer's past behavior. The different scenarios identified in the model can be related to various assumptions about the tax administration's monitoring technology.
\begin{enumerate}[topsep=3pt]
\item The scenario where information revelation is unconditional corresponds to a situation where the tax administration is not allowed to condition the disclosure of information on available
information on the taxpayer’s past behavior or the resulting conviction.
This scenario also serves as a fundamental benchmark for comparison with the subsequent two scenarios.

\item 
  The scenario where information provision can be contingent on the taxpayer's past behavior arises when a violation becomes known to the tax administration, even if it does not lead to prosecution and punishment (indicated in our model by a negative public signal $\ell$) of the offender.\footnote{This scenario may occur for different reasons, including the inherent limitations of evidence and the higher burden of proof required for convicting and punishing offenders compared to mere knowledge of a violation,
  or the limitation of resources available to the tax administration to start a formal inspection in all cases where evidence available points to a likely violation. Law enforcement typically allocates more resources towards investigating and apprehending potential offenders whom they have reason to believe engaged in the past in unlawful activities, even if these individuals have not yet been convicted or subjected to punishment.}

\item 
  The scenario where information provision is determined solely by the public signal (i.e., the actual conviction of the taxpayer in past periods) is relevant when the tax administration cannot determine if a violation has occurred, unless a full inspection has been conducted and the taxpayer's past behavior has been verified and established.
\end{enumerate}

By investigating these scenarios, we provide valuable insights that may improve the strategies of law enforcement agencies, thereby potentially reducing wrongdoing and enhancing public goals.

\paragraph{International Relations.}
Another application of our model is in the realm of international relations. Abstracting away from exact details, consider a repeated interaction between two countries X and Y. Country X contemplates taking an action that could be perceived as hostile by country Y, potentially triggering a negative reaction. Suppose that the probability of negative reaction can be either high or low depending on country's Y type, which is a private information. Country X benefits from taking the action only if the probability of a reaction from Y is low, otherwise it suffers harm. However, X lacks complete information about Y's likelihood of reacting. The repeated nature of the interaction implies that country X may engage in an explore-exploit strategy. Now imagine that country Y or a third country Z, which knows the type of country Y, have a goal to discourage country X from taking the hostile action. As an alternative to a direct reaction that might escalate to irreversible and potentially catastrophic outcome, country Y (or country Z), playing the role of the sender, could utilize diplomatic channels or informal intelligence relations to provide information to X, playing the role of the receiver, about the probability of Y's reaction (the type of country Y).  The information might encompass insights into Y's readiness for conflict, the attitudes of key decision makers within Y's government, or the perceptions of influential players in the international community.

The timing and content of this information disclosure can have a significant impact on X's decision-making process. Repeated interactions between the sender and receiver can influence X's incentive to ``test'' Y's reaction through its actions. Moreover, the provision of information can be made contingent on specific events, such as X's actions.\footnote{\citet{hennigs.21} analyzes the role of information disclosure by a mediator who wants to prevent a conflict in a model of Bayesian persuasion. However, that model considers a single period and the parties do not learn from their interaction.}

\paragraph{Parents and children.}  
Children and young adults often consume substances or participate in activities that might cause them harm. Examples include using drugs, engage in unsafe sex, driving in motorbikes, swimming in unauthorized beaches, and more.
These activities are often repeated, allowing the teenager to learn the associated risk level and whether they would like to continue participating in them. It stands to reason that the notion that teenagers are risk takers at least partly can be understood as a story of utilizating an explore-exploit strategy.\footnote{Indeed, the literature  has studied how children take risky decisions and whether they have Bayesian intuition, see,
e.g., \citet{jacobs1991use}, \citet{reyna1994fuzzy}, \citet{levin2003risk}, \citet{gigerenzer2021children}.} 

Parents are sometimes more informed about the riskiness of the activities or the ability of their children to safely participate in these potentially harmful activities. Moreover, some parents may wish to dissuade their children from participating in these activities, irrespective of their actual risks, due to perceiving them as excessively dangerous or harmful. These parents may achieve their goals by prohibiting their children from engaging in these activities and punishing them if they disobey or by offering rewards to the children if they behave in the desired way. Alternatively, parents can intervene in a less drastic way by offering information about the riskiness of the activity.  

Our results suggest that the optimal strategy for parents depends on the monitoring structure. When parents are well informed of their children behavior, 
the optimal way to dissuade children is a rewarding strategy --- the child is rewarded for not participating in the risky activity.
A rewarding strategy is also optimal when parents learn that their children participation in the activity only if it goes awry. On the other hand,
in cases where the parents will never know whether their children participate in the risky activity, even if it goes amiss,
they  should provide information about the riskiness before the children decide whether to participate in the activity.

\paragraph{Gambling.} Lastly, our model applies to the original ``one-armed" bandit problem. In the classical bandit problem, a gambler is faced with the decision of playing a slot machine repeatedly. The outcome of playing is contingent on the unknown probability of success, which could be either high or low, resulting in either a gain or a loss to the gambler. The decision-maker lacks knowledge of the exact probability of success, but holds a prior belief about it, and therefore he may engage in explore-exploit strategy. 

Imagine now an additional entity, such as a non-governmental organization or the government, who possesses information about the likelihood of success and is committed to reducing the use of slot machines. The government, for example, could of course prohibit gambling and punish violators. However, in recent years there is a tendency in public policy to utilize less drastic tools to influence behavior. One possibility for the government is to provide information to the gambler about the probability of success. It stands to reason that if the probability of success is low, the government would like to inform the gambler. But what should the government do if the probability of success is high? Our analysis can inform the government on when and how much information it should disclose in order to fulfill the objective of limiting slot machine usage.


\section{Benchmark: No Information Provision}
\label{sec: no information}

We examine here a benchmark scenario where the sender provides no information to the receiver. Therefore, the receiver may update his belief about the real state of nature only after taking action A and receiving a signal. 

Since in any model with information provision the receiver has the option to disregard any information the sender provides, the payoff associated with the optimal strategy for the receiver in this benchmark case is a lower bound on the payoff that the receiver will attain in \emph{any} Stackelberg equilibrium, even when the sender can send messages to the receiver.  Since the sender's cost when the receiver selects A is 1, and since the receiver's payoff is at most 1, the receiver's payoff in the benchmark case is also a lower bound on cost to the sender in all Stackelberg equilibria when messages can be sent.

Without information, among all possible strategies in the two-stage decision problem, the receiver has three optimal strategy candidates:%
\footnote{It is not difficult to verify that any other strategy is weakly dominated by one of these three strategies.} 
\begin{enumerate}[label=(\roman*)]
\item RR: choosing R in both periods, which yields payoff zero;
\item AC: choosing A in period 1 and making a conditional choice in period 2, namely, choosing A in period 2 if and only if the signal in period 1 was positive, which yields payoff $q-(1-q)c+\a_Gq-\a_L(1-q)c$;
\item AA: choosing A in both periods, which yields payoff $2(q-(1-q)c)$.
\end{enumerate}

When the prior probability of $G$ is sufficiently low (say, $q=0$), the optimal strategy is RR. Conversely, when it is sufficiently  high (say, $q=1$), the optimal strategy is AA. Strategy AC is optimal in an intermediate range characterized by two cut-off points:
\begin{eqnarray}
\label{equ:q1}
q^i&:=&\frac{c(1+\a_L)}{(1+\a_G)+c(1+\a_L) },\\
\label{equ:q2}
q^{ii}&:=&\frac{c(1-\a_L)}{(1-\a_G)+c(1-\a_L)}.
\end{eqnarray}
so that the optimal receiver's strategy can be characterised as follows:
\begin{proposition}[\textbf{No Information Provision}]
\label{theorem:optimum}
When the sender provides no information, the optimal strategy for the receiver is:
 \begin{enumerate}
 \item play RR when $0\leq{q}\leq{q^i}$;
 \item play AC when $q^i< {q}\leq{q^{ii}}$;
 \item play AA when ${q^{ii}} < {q}\leq{1}$.
 \end{enumerate}
\end{proposition}

The receiver's optimal strategy that is indicated in Proposition~\ref{theorem:optimum} induces the following payoff $\pi(q)$ to the receiver and cost $N(q)$ to the sender:
\begin{equation}
\label{eq: optimal strategy}
   \pi(q) :=
  \begin{cases}
    0, &0<q\leq{q^i},\\
   \big(q-(1-q)c+q\alpha_G-(1-q)\alpha_Lc\big), &q^i<q\leq q^{ii},\\
     2(q-(1-q)c), &q^{ii}<q\leq{1},
  \end{cases}
\end{equation}

\begin{equation}
  N(q) :=
  \begin{cases}
    0, &0<q\leq{q^i},\\
   1+q \a_G+(1-q)\a_L, &q^i<q\leq q^{ii},\\
     2, &q^{ii}<q\leq{1}.
  \end{cases}
\end{equation}

Figure \ref{fig:1} depicts these two functions, 
$\pi(q)$ in black and $N(q)$ in red; $\pi(q)$ is the maximum of three lines representing the payoff from the three strategies RR (between 0 and $q^i$), 
AC (between $q^i$ and $q^{ii}$)
and AA (between $q^{ii}$ and 1). 

\begin{figure}
\centering
\includegraphics[width=.75\linewidth]{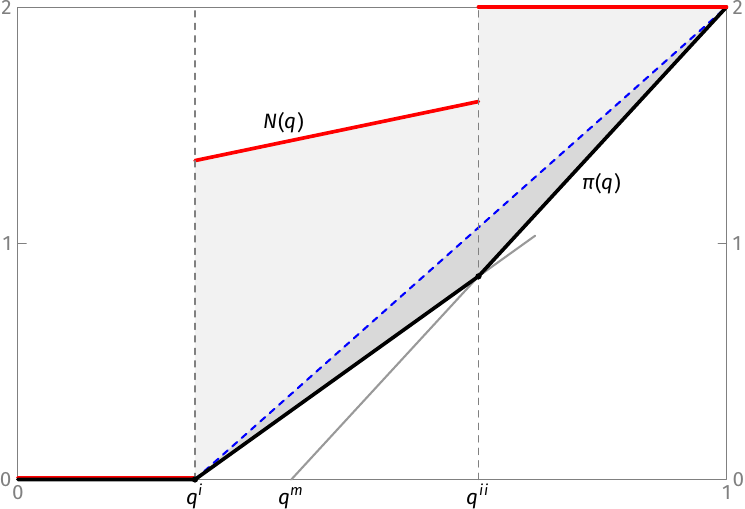}
\caption{Receiver's optimal payoff (black) and sender's cost (red).}
\label{fig:1}
\end{figure}

We denote by $q^m$ the level of $q$ at which the receiver's one-stage payoff from taking action A, i.e., $q-(1-q)c$, is zero:
\begin{equation}\label{qm}
    q^m := \frac{c}{1+c}.
\end{equation} 

It is easy to verify that $q^m \in (q^i,q^{ii})$.  Observe that in period 2 the receiver chooses R (respectively, A) if his posterior belief is below (respectively, above) $q^m$.  When $q > q^{ii}$, the posterior belief is above $q^m$ under both signals, and AA is the optimal strategy of the receiver.

When $q \in (q^i,q^{ii})$, the posterior belief of the receiver is above $q^m$ if he receives the positive signal, and below $q^m$ if he receives the negative signal, and AC is the optimal receiver's strategy.  We can further divide $(q^i,q^{ii})$ into two sub-intervals, $(q^i,q^m)$ and $(q^m,q^{ii})$.  In $(q^i,q^m)$ the receiver strategy is ``explore-exploit'': he bears a negative payoff in period 1 in exchange for a larger positive payoff in period 2.
In $(q^m,q^{ii})$ the receiver strategy is ``exploit-learn-exploit'': the payoff from A is positive in period 1, and in addition, it generates valuable information, such that in period 2 the receiver acts if and only if the signal is positive.

When $q < q^i$, acting yields a negative payoff, and the only incentive to act in period 1 is to obtain information on the state. In this case, even when the posterior after obtaining the positive signal is above $q^m$, the gain in period 2 does not compensate for the loss in period 1, implying that RR is the optimal strategy for the receiver.

\paragraph{The function $\pi(q)$ as a benchmark. }
As mentioned above,
$\pi(q)$ is a lower bound on the receiver's payoff and on the sender's loss in any Stackelberg equilibrium,
even when the sender can send information. 

In a hypothetical scenario where the sender observes the receiver's choices and can make monetary transfers to the receiver, transferring $\pi(q)$ to the receiver would be sufficient to prevent the receiver from ever choosing
A. 
Yet, in our framework, direct monetary transfers from the sender to the receiver are not allowed. Instead, the sender can ``pay'' the receiver with information, which the receiver utilizes to optimize his actions and payoffs. This, of course, is costly to the sender. 

The difference $N(q)-\pi(q)$, which is reflected by the vertical distance between the red line and the black line in Figure~\ref{fig:1}, measures the maximal potential gains for the sender in providing information. 
If the sender can send messages to the receiver and if $\widetilde N(q)$ denotes the sender's cost in equilibrium, then the difference $\widetilde N(q) - \pi(q)$ measures the sender's loss due to the information structure and her inability to transfer money to the receiver. 
The closer the cost to the sender is to $\pi(q)$, the more effective the sender's information provision policy is in influencing the receiver's decision.

\section{Unconditional Information Provision}\label{sec: uncoditional}

In this section we examine the scenario where the information the sender provides to the receiver is not conditioned on the receiver's actions or on the signals.  Our goal is to determine the optimal revelation strategy, namely, when and how much information to reveal, and its effectiveness in reducing the sender's costs. Since with no information the cost to the sender is $0$ when $q\in(0,q^i]$, we restrict attention to values $q\in(q^i,1)$.

The set of messages $M$ is assumed to contain two messages, which we denote for convenience $g$ and $\ell$, so that $M=\{g,\ell\}$.

\paragraph{Messaging strategy.}  
Since the sender's strategy is independent of the receiver's behavior,
the sender's strategy in period 1 is a function $\mu_1 : [0,1] \times \{G, L\} \rightarrow \Delta(\{g,\ell\})$,
where $\Delta(\{g,\ell\})$ is the set of probability distributions over $\{g,\ell\}$.
The sender's strategy in period 2 can depend on the message sent in period 1, hence it is a function 
$\mu_2 : [0,1] \times \{G, L\} \times \{g,\ell\} \rightarrow \Delta(\{g,\ell\})$.

Let $0\leq r < q < 1$. We say that the sender \emph{splits $q$ into $\splitted{r;1}$} if she uses the following strategy:
\begin{enumerate}[label=(\roman*)]
\item When the state of nature is $G$, the sender sends the message $g$ with probability $\gamma=\frac{q-r}{(1-r)q}$, and the message $\ell$ with probability $1-\gamma$.
\item When the state of nature is $L$, the sender always sends the message $\ell$.
\end{enumerate}

Simple algebraic manipulations show that, when the receiver's prior belief is $q$, upon obtaining a message, the receiver's updated belief is
\begin{equation}
  \mathbb{P}(G\mid g)=1 \quad\text{and}\quad 
  \mathbb{P}(G\mid\ell)=r,
  \label{eq: split}
\end{equation}
which explains the appellation ``split $q$ into $\splitted{r;1}$''.
Moreover, the probability of the signal $g$ (resp., $\ell$) is $q \gamma$ (resp., $1-q\gamma$).

We will also say that the receiver responds with the ``compliance strategy" if he acts in accordance with the information the sender provides: when the sender reveals that the state is $G$, the receiver acts, and otherwise he refrains. 

The following result characterizes the Stackelberg equilibrium.

\begin{proposition}[\textbf{Unconditional Information Provision}]
\label{theorem: unconditional}
When information revelation is unconditional on actions or signals, for any prior $q$, Stackelberg equilibrium payoffs are unique. The following represents an equilibrium:
\begin{enumerate}
\item When $0<q\le q^{i}$, the sender reveals no information and the receiver plays $RR$. The sender's cost in this equilibrium is $\N{U}(q)= 0$.
\item When  $q^{i}<q<1$, the sender reveals information only in period 1. Furthermore, at the outset of the game the sender splits $q$ into $\splitted{q^{i};1}$. The
receiver complies: he responds with RR if the split realization is  $q^{i}$ and with AA if it is $1$. The sender's cost in this equilibrium is 
$\N{U}(q) = 2\,\frac{q-q^i}{1-q^i}$. 

\end{enumerate}
\end{proposition}

Proposition \ref{theorem: unconditional} implies that it is inefficient for the sender to provide information unconditionally in period 2. In equilibrium, any disclosed information should be provided before the receiver can take any action (in period 1). 
When the prior belief $q$ is smaller than a certain threshold (i.e., $q^{i}$),  without any additional information, the receiver has no incentive to act. In this trivial case, there is no reason to disclose information.
When $q>q^{i}$, in equilibrium, if the state is $G$, the sender will reveal the truth to the receiver with a positive probability (which is a function of $q$). 
In this case, the receiver will respond with AA. 

In all other cases, the receiver will assign probability $q^i$ to $G$, and will respond with RR. 
The cost to the sender is, therefore, twice the probability with which revelation occurs. 
The blue dashed line in Figure~\ref{fig:1} exhibits the equilibrium cost to the sender. As can be seen, it is below the red line, but above the black line. In words, with Unconditional Information Provision, the sender's cost in the Stackelberg equilibrium is lower than the cost with no information provision, but still does not reach the lower bound.

At first glance, it may seem obvious that the sender is better off disclosing information upfront in period 1 rather than delaying it until period 2. However, upon closer examination, one might question why the sender would not simply commit to revealing unconditional information in period 2 that coincides with what the receiver could discover by acting in period 1. Such a commitment could seemingly incentivize the receiver to refrain in period 1, saving the associated costs. 
As Proposition \ref{theorem: unconditional} demonstrates, this reasoning is flawed.
The fallacy lies in the assumption that acting in period 1 would fully reveal the truth. Due to the stochastic nature of the signal, acting in period 1 will inevitably leave room for further valuable exploration.

The receiver can always act in period 1, and acquire additional and independent information beyond any (non-fully revealing) unconditional information revealed in period 2. This implies that the sender may never be able to fully replicate the information the receiver can obtain in period 1.  In cases where the prior probability $q$ falls within the range $(q^{i},q^m)$, a reasonable amount (from the sender's perspective) of unconditional disclosure of information in period 2 would still leave the receiver's benefit of acting in period 1 outweighing the associated cost.

Stated differently, when information revelation is unconditional on actions or signals, the receiver possesses an additional strategy not available in the benchmark case: selecting A in period 1 and conditioning the choice in period 2 on both the signal obtained in period 1 and the message provided by the sender. The non-trivial assertion of Proposition \ref{theorem: unconditional} is that the amount of information revealed in period 2, which could entice the receiver to forego the exploration cost in period 1, is so substantial that it exceeds the amount necessary to achieve this outcome from the start, even considering the potential value of this information in both periods.

\paragraph{Uniqueness.}
In Proposition \ref{theorem: unconditional}, it is stated that the equilibrium payoffs are unique. However, the equilibrium strategies are typically not unique. Specifically, for values of $q$ that fall within a range  where the sender's payoff is linear, the sender has the flexibility to split the prior probability into different values within this interval without affecting the payoffs. Therefore, when the cost to the sender is linear, the sender can choose from multiple equilibrium strategies. For instance, in the trivial case of Proposition \ref{theorem: unconditional}, when $q\in (0,q^i)$, the sender can split $q$ into $\splitted{0;q^i}$ at the beginning of the game without changing the outcome. 


\section{Action-Based Information Provision}\label{sec: action-based}
We now examine the scenario where the sender's messaging policy can be based on the receiver's behaviour in period 1. 
 Formally, the sender's messaging strategy is $\mu = (\mu_1, \mu_2) $, where  $\mu_1: [0,1] \times \{G, L\} \rightarrow \Delta(\{g,\ell\})$

and $\mu_2: [0,1] \times \{G, L\} \times \{\text{A}, \text{R}\}\rightarrow \Delta(\{g,\ell\})$.

Suppose the sender reveals the true state of nature in period 2 and the receiver utilizes the compliance strategy: he selects R in period 1 and A in period 2 if and only if the sender revealed that the state is $G$. Then the payoff to the receiver is $q$, namely, the identity function. The following lemma relates the identity function with $\pi(q)$, the lower bound on the receiver's equilibrium payoff.

\begin{lemma}\label{lemma 1}
There is a unique point $p^*>0$ such that $\pi(p^*)=p^*$. Moreover,
\begin{enumerate}[label=\alph*)]
\item $p^* \geq q^{ii}$  if and only if $\,2\a_G - \a_L \le 1$;\label{lemma:a}
\item $p^*>q^m$;\label{lemma:b}
\item if $q<p^*$ then $q \leq \pi(q).$\label{lemma:c}
\end{enumerate}
\end{lemma}

Lemma~\ref{lemma 1} states that, depending on the  primitives of the model, $p^*$  might be above or below $q^{ii}$. The following theorem emphasizes the importance of this relationship to achieving the lower bound of the sender's loss (the lower bound is attainable only when $p^* \geq q^{ii}$). 

Denote
\begin{equation}
    \xi(q):=\frac{q-\pi(q)}{1-\pi(q)}.
\end{equation}
By Lemma \ref{lemma 1}(c),  $\xi(q)\geq0$ whenever $q\leq p^*$.

\begin{theorem}[\textbf{Action-Based Information Provision}]
\label{theorem:observe}
When the messaging strategy of the sender may depend on the  receiver's action, equilibrium strategies, the equilibrium payoff of the receiver, and the sender's cost, denoted $\N{A}(q)$, are as follows:
\begin{enumerate}
\item For $0<q\le q^{i}$, the sender reveals no information and the receiver plays $RR$. 
In this case, $\N{A}(q)= 0$. 
\item For $q^i<q\leq p^*$, the sender provides no information in period 1; in period 2, conditional on the receiver choosing R in period 1, she provides information by splitting $q$ into $\splitted{\xi(q);1}$, where $\xi(q)<q^i$. The receiver follows the compliance strategy.
In this case, $\N{A}(q)=\pi(q)$.
\item For $p^*<q <1$,
the sender splits $q$ into  $\splitted{p^*;1}$ in period 1; if the split realization is $p^*$ and the receiver refrains in period 1, then in period 2 she splits $p^*$ into $\splitted{0;1}$.
In this case, $\N{A}(q)=\pi(q)$ if and only if $p^*> q^{ii}$. Otherwise, $\N{A}(q)>\pi(q)$.
\end{enumerate}
\end{theorem}

Theorem \ref{theorem:observe} states that it is efficient to condition information disclosure on the actions of the receiver, leading to the optimal messaging strategy of deferring information disclosure to period 2 \citep[see also][]{fuchs2007contracting}. This finding sharply contrasts with Proposition \ref{theorem: unconditional}, where it is efficient not to delay information.

 When the state $G$ is completely revealed, the maximal payoff to the receiver in period 2 is $q$. Moreover, $\pi(q)$ is a lower bound on the receiver's equilibrium payoff. It implies that when $q \geq \pi(q)$, providing information in period 2 alone suffices to incentivize the receiver to choose R in period 1. Conversely, when $q < \pi(q)$, postponing information disclosure to  period 2 becomes insufficient, requiring that information be provided also in period 1. According to Lemma~\ref{lemma 1}, the threshold level, denoted  $p^*$, determines whether the information in period 2 is sufficient to dissuade the receiver from choosing R in period 1.

When $p^* > q^{ii}$, providing information in period 1 does not harm the sender's interests. In fact, it allows the sender to achieve the lower bound of her loss. This is because, when $p^* > q^{ii}$, within the interval $(p^*,1)$, the receiver's optimal strategy in the absence of information is to choose strategy AA. Therefore, the additional information revealed by the sender in period 1 does not improve the receiver's payoffs (refer to Figure \ref{fig:2}(b)).

 Conversely, when $p^* < q^{ii}$, within the interval $(p^*,q^{ii})$, the receiver's optimal strategy without any information is to choose strategy AC. However, the receiver can utilize the information provided in period 1 to increase his payoffs by employing a conditional strategy: choosing strategy AA when state $G$ is revealed in period 1, and choosing strategy AC when the state $G$ remains undisclosed. This is the reason why providing information in period 1 becomes detrimental to the sender. In this case, the sender does not attain the lower bound of her loss (as depicted by the dashed blue line in Figure \ref{fig:2}(a)).
  
\begin{figure}
\makebox[\textwidth][c]{
\begin{minipage}{.6\textwidth}
\centering
\includegraphics[width=\linewidth]{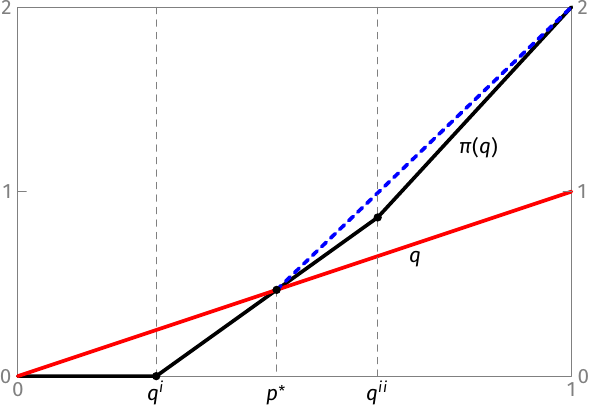}
(a) The case $p^*<q^{ii}$
\end{minipage}
\hspace{2mm}
\begin{minipage}{.6\textwidth}
\centering
\includegraphics[width=\linewidth]{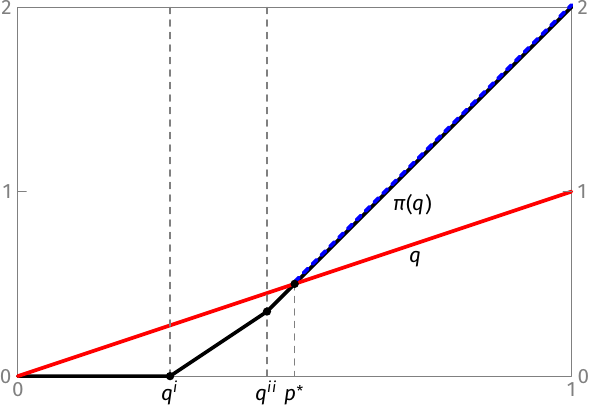}
(b) The case $p^*\ge q^{ii}$
\end{minipage}}
\caption{The graphs of $q$ (red) and $\pi(q)$ (black),
together with the sender's cost in equilibrium when $q>p^*$ (blue, dashed).}
\label{fig:2}
\end{figure}

\paragraph{Uniqueness.}
In Theorem \ref{theorem:observe}, like in Proposition \ref{theorem: unconditional}, equilibrium payoffs are unique, but equilibrium strategies need not be unique. For example, when $p^*<q^{ii}$ and $q\in(q^{i},p^*)$, the sender can split the prior to values within the interval $(q^i,p^*)$, and promise to reveal the appropriate level of information in period 2, conditional on the receiver choosing R in period 1.

\section{Signal-Based Information Provision}
\label{sec: signal-based}

In this section we assume that information provision cannot be based on the actions of the receiver, but rather on the signal they generate. 
This amounts to assuming that actions are private and the signals are public. 
In this variation, the sender's messaging strategy is $\mu = (\mu_1, \mu_2)$, where $\mu_1: [0,1] \times \{G, L\} \rightarrow \Delta(\{g,\ell\})$ and $\mu_2: [0,1] \times \{G, L\} \times \{+,-\} \rightarrow \Delta(\{g,\ell\})$.

The next result describes the Stackelberg equilibrium.
As it turns out, the equilibrium often has the same structure as in the action-based scenario.

\begin{theorem}[\textbf{Signal-Based Information Provision}]
\label{theorem:signal-based}
Suppose the receiver's actions are private, but the signals they generate are public. Then, an equilibrium strategy, the equilibrium payoff of the receiver, and the sender's cost, denoted $\N{S}(q)$, are as follows: 
\begin{enumerate}
    \item If $\a_G\le \frac{1}{2}$, then the equilibrium coincides with the action-based equilibrium as described in Theorem~\ref{theorem:observe}.
    \item If $\a_G> \frac{1}{2}$, there is a threshold 
    $p^e\in(q^m,\min{[p^*, q^{ii}]})$  such that:
    \begin{enumerate}[label=\roman*)]
        \item For $0 < q\le p^e$, the equilibrium coincides with the action-based equilibrium, as described in cases $a$ and $b$ of Theorem \ref{theorem:observe}. In this case, $\N{S}(q) = \pi(q)$.
        \item For $p^e<q<1$, in period 1 the sender splits $q$ into $\splitted{p^e;1}$. If the split realization is $p^e$, then in period 2 she splits $p^e$ into $\splitted{\xi(p^e);1}$. In this case, $\N{S}(q)>\pi(q)$.
     \end{enumerate}
\end{enumerate}
\end{theorem}



Theorem 2 highlights that in certain circumstances, that is,
when $\alpha_G\leq 1/2$ or $q\leq p^e$, allowing the sender to base information on actions does not confer any practical advantage over relying solely on signals.

To gain some understanding for this result, note that when messages depend on signals, the receiver will obtain exactly the same information from the sender if he chooses R in period 1. However, the receiver does not need to choose R in period 1 to obtain the message later. In fact, even after choosing A, there is a positive probability of obtaining the information from the sender. 
This probability depends on several factors, including $q$, $\alpha_G$, and $\gamma$, reflecting the splitting probability. In situations where $\alpha_G<1/2$ or when $q$ is relatively low, the probability, and hence the value of receiving an informative message after selecting A, 
are particularly low. In such cases, it is better for the receiver to disregard the message conveyed by the sender, and instead rely in period 2 solely on the signal obtained by his own actions. 
Indeed, in the signal-based scenario, if the receiver chooses A in period 1, he can either use the signal generated by his action and disregard the message delivered by the sender, or conversely, rely on the message from the sender and ignore the self-generated signal, but he cannot use both the signal and the message.\footnote{This stands in sharp contrast to the Unconditional Information Provision scenario, where messages from the sender can be utilized together with signals generated by actions, and from the Action-Based Information Provision scenario, where messages are never sent if the receiver chooses A in period 1.} 
When the receiver chooses A in period 1 and is better off ignoring the message in period 2 and relying on his signal, he obtains the payoff $\pi(q)$.
To discourage the receiver from choosing A in period 1, the sender must provide conditional information that results in a payoff to the receiver of at least $\pi(q)$. This constraint is precisely what the sender faces in the action-based scenario. Furthermore, as pointed out, when the receiver abstains from choosing A in period 1, the information (message) obtained remains identical regardless of whether it is based on actions or signals. Consequently, to satisfy this constraint, the sender must employ the same strategy. 

When $\alpha_G>1/2$ and $q > p^e$, the Signal-Based Information Provision is still more effective than the Unconditional Information Provision, but less effective than the Action-Based Information Provision. 
This is so precisely because in the signal-based scenario the receiver can obtain and utilize the information promised by the sender even when he chooses A in period 1, something which is impossible in the action-based scenario. 

In both the action-based and the signal-based scenarios, the sender chooses to defer information revelation to period 2 and, if the promise of future information alone is insufficient to influence the receiver's behavior in period 1, the sender provides additional information in period 1. 

In all scenarios, when the sender opts to disclose information in period 1, she splits the prior between the upper limit of 1 and a certain level $q^*$. The magnitude of $q^*$ varies across the different scenarios, providing interesting insights. Specifically, the unconditional scenario exhibits the lowest value of $q^*$, namely $q^i$, reflecting the highest level of information disclosure, and consequently the highest loss to the sender. On the other hand, the action-based scenario demonstrates the highest value of $q^*$, namely $p^*$, indicating the lowest level of information disclosure, and consequently the lowest loss to the sender. Accordingly, we observe the ordering $q^i < p^e < p^*$, which reflects the varying levels of $q^*$ among the three scenarios when the sender opts to disclose information in period 1. 

Figure \ref{fig:final} illustrates  this ordering in the case where $\alpha_G>1/2$ and $p^*<q^{ii}$. Here, the dashed black line, reflecting the unconditional scenario split, is above the red line which reflects the signal-based scenario split in period 1. The latter is further above the blue line, which represents the action-based scenario split in period 1.

\begin{figure}
\centering
\includegraphics[width=.75\linewidth]{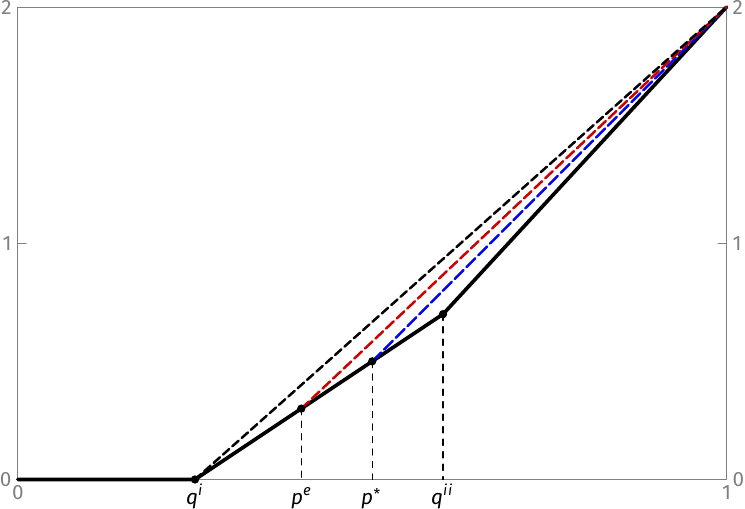}
\caption{Optimal splitting with Action-Based (dashed blue line), Signal-Based (dashed red line) or Unconditional Information (dashed black line)}
\label{fig:final}
\end{figure}

A further difference between the action-based and the signal-based scenarios is that in the former case providing information in period 1 is always accompanied by full revelation in period 2 (formally, $p^*$ is split into $\splitted{0;1}$).  However, in the signal-based scenario, the situation is different as long as the equilibrium strategies do not coincide. Here, in period 1 the prior is split into  $\splitted{p^e;1}$. When the receiver's belief in period 2 is $p^e$, since $p^e$ is always less than $p^*$,  the optimal revelation by the sender still leaves the receiver with some residual uncertainty regarding the actual state. This uncertainty is formally represented by the split of $p^e$ in period 2 into $\splitted{\xi(p^e);1}$ (since $p^e<p^*$, $\xi(p^e)\geq 0$).


\section{Summary}
\label{sec:conclusion}
\paragraph{An overview.}
In this paper we incorporate Bayesian persuasion into a bandit exploration problem.
We analyse a dynamic persuasion model that accounts for the possibility that the receiver partially learns about the true state of nature through his actions, which may harm the sender. In this model, the receiver considers not only the payoff, but also the informational value of his actions. The sender's objective is to minimize the number of times the receiver chooses the harmful (and informative) action. We distinguish three scenarios. The first is where the sender provides information unconditionally. The second is where the sender can condition the information provision on the actions of the receiver, 
and in the third scenario the sender can condition information based only on the public signals. Table \ref{tab:summary} summarises our main results.

\definecolor{lightgray}{gray}{0.95}
\newcommand{\conditional}%
  {\begin{tabular}[c]{l}R if $\mu_2=\ell$\\ A if $\mu_2=g$\end{tabular}}
\begin{table}
\footnotesize
\caption{A summary of equilibrium strategies}
\label{tab:summary}
\makebox[\textwidth][c]{
\begin{tblr}{width=1.2\linewidth,
      colspec={|Q[m,l,.6]Q[m,r,.5]|
      Q[m,l,.5]|Q[m,l,.5]|Q[m,l,.5]|Q[m,l,.5]|
      Q[.3,m,l]|},
    rows={ht=5mm},hlines,hline{4,7,9,12,16}={1.5pt},
    column{2-7}={c},cell{5,8,11,14}{7}={lightgray}}
\SetCell[r=3]{b} Variation
& \SetCell[r=3]{b}Prior Range
& \SetCell[c=4]{c}Strategies
&&&
& \SetCell[r=3]{b}Lower Bound $N=\pi$ \\
& & \SetCell[c=2]{c}Sender&& \SetCell[c=2]{c}Receiver && \\
& & period 1 & period 2 & period 1 & period 2 & \\
\SetCell[r=3]{m}No Information (benchmark)
& $0< q\le q^i$
& \SetCell[r=3,c=2]{c,m} No Information &
& R & R
&  Yes \\
& $q^i< q\le q^{ii}$
&& 
& A & {R if signal ($-$)\\ A if signal ($+$)}
& \SetCell[r=2]{m} No \\
& $q^{ii}< q< 1$
&&
& A & A
&  No \\
\SetCell[r=2]{m}Unconditional Information Provision
& $0<q\le q^{i}$
& \SetCell[c=2]{c} No Information &
& R & R
&  Yes \\
& $q^{i}<q<1$
& {Split $q$\\ into $\splitted{q^{i};1}$} & No Information
& \SetCell[c=2]{c}  {RR if $\mu_1=\ell$, AA if $\mu_1=g$} &
& No \\
\SetCell[r=3]{m}Action-Based Information Provision
& $0 < q\le q^i$
& \SetCell[c=2]{c} No Information &
& R & R 
& Yes \\
& $q^i<q \le p^*$
& No Information & {(*) Split $q$ \\ into $\splitted{\xi(q);1 }$}
& R & {R if $\mu_2=\ell$\\ A if $\mu_2=g$}
& Yes \\
& $p^*<q < 1$
& {Split $q$\\ into $\splitted{p^*;1 }$} & {(*) Split $p^*$\\ into $\splitted{0;1}$}
& \SetCell[c=2]{c} {R if $\mu_1=\ell$, then \conditional\\ AA if $\mu_1=g$ }&
& Only if $p^*>q^{ii}$ \\
\SetCell[r=3]{m} Signal-Based Information Provision ($\alpha_G> 1/2$)
& $0<q\le q^i$
& \SetCell[c=2]{c} No Information &
& R & R
& Yes\\
& $q^i<q\le p^e$
& No Information & {(**) Split $q$\\ into $\splitted{\xi(q);1}$}
& R & \conditional
& Yes\\
& $p^e<q<1$
& {Split $q$\\ into $\splitted{p^e;1}$} & {(**) Split $p^e$\\ into $\splitted{\xi(p^e);1}$}
& \SetCell[c=2]{c} {R if $\mu_1=\ell$, then \conditional\\ AA if $\mu_1=g$}&
& No\\
($\alpha_L\le 1/2$)
& $0< q \le 1$
& \SetCell[c=5]{c} Like the Action-Based Information Provision\\
\end{tblr}}
\smallskip

\noindent (*) = conditional on the receiver choosing R in period 1.\\
\noindent (**) = conditional on the signal in period 1 being positive.
    
\end{table}

Our findings indicate that, among the three scenarios, the best equilibrium outcome, from the sender's perspective,  occurs when she can condition the information she provides on the actions taken by the receiver. The worst equilibrium outcome is in the case where the sender cannot condition her information provision on anything, except for the prior.

Although the scenarios of the Action-Based and the Signal-Based Information Provision may differ in terms of efficiency, they share a common feature. Namely, in both scenarios, depending on the prior beliefs, the optimal provision occurs either exclusively in period 2 or in both periods, but it never occurs solely in period 1. In contrast, in the model where the sender cannot condition information on the receiver's behavior or on the public signal, information is provided only in period 1, and its effectiveness is accordingly limited.

Despite notable differences, our analysis has unveiled several features shared by all the scenarios considered. In each model, the optimal strategy prescribes both the quantity of information provided and the timing of its provision.
In all Stackelberg equilibria, the sender employs a ``clear-message" strategy, while the receiver adopts a ``compliance" strategy. This entails that the sender, with a positive probability, discloses that the state of nature is $G$ or maintains a sufficient level of uncertainty regarding the true state of nature. 
In response, the receiver only takes the detrimental action when the sender reveals that the state of nature is $G$.

\paragraph{Exploratory opportunities are never exploited.} A noteworthy characteristic of our model, independent of what the information provided may rely upon, is that the receiver possesses exploration opportunities. Specifically, the receiver can partially learn about the true state of nature by taking the harmful action in period 1. Surprisingly, in all equilibria, the receiver does not exploit this opportunity. The reason is that the sender commits to disclosing an amount of information that incentivizes the receiver either to refrain in period 1, or to act in both periods when the state $G$ has already been revealed to him by the sender and nothing has been left unknown. Consequently, the receiver relies solely on the information provided by the sender.

\color{black}

\paragraph{Policy implications: law enforcement.}
Our model applies naturally to many two-period one armed-bandit problems, where there is another player who aims at discouraging the decision maker from taking certain actions. 
One primary application of our study is
law enforcement in general and tax evasion in particular.  

Our findings regarding Action-Based and Signal-Based Information Provision have important implications for tax enforcement policy, and possibly law enforcement policy in general. 
They suggest that, when policies rely on strategies of enforcement communication, authorities should refrain from disclosing full information about enforcement at the outset. Instead, whenever conditional disclosure of information is feasible, a more effective approach would be to gradually reveal information based on the behavior or records of offenders. This strategy implies that individuals with a clean record should be entitled to receive information about enforcement, while those with a negative record should not.

The concept of Signal-Based Information Provision resembles other aspects of law enforcement; 
specifically, the principle that repeat offenders undergo heightened scru\-ti\-ny and receive more severe punishments. For instance, the U.S. Federal Sentencing Guidelines state that ``a defendant with a history of prior criminal behavior is more culpable than a first-time offender and therefore deserves a greater punishment.'' In our model, non-repeat offenders are not granted lighter sentences or reduced enforcement. Instead, they are ``rewarded'' with valuable information about enforcement. By implementing this approach, enforcement authorities can effectively utilize information provision as a tool to incentivize good behavior and deter potential offenders. 


\color{black}

\paragraph{Information provision vs.\ monetary transfers.} 
Information provision contrasts with incentives provided through monetary transfers, where similar timing issues are not typically observed. This distinction arises from two key differences.
Firstly, monetary transfers can generally be executed in a single stage, regardless of their magnitude. This eliminates the need for strategic timing considerations.
Secondly, monetary transfers lack the ability to influence future behavior; they have no lasting impact on the receiver's decision-making process. As a result, early and unconditional monetary transfers are often ineffective. In contrast, information disclosure alters the receiver's beliefs, thereby affecting his future actions.

\paragraph{Future directions.}
Our analysis focused on an interaction between two players: a sender and a receiver, akin to an enforcer and an offender. A natural extension of this analysis involves scenarios with one sender and multiple receivers. In such settings, a crucial consideration is whether the sender can disclose private information to individual receivers or must provide public information accessible to all receivers. 

In the case of private information disclosure, our findings extend seamlessly. The sender can strategically tailor the information provided to each receiver, potentially influencing their decisions and achieving the desired outcome. 
When public information is the only option, the dynamics become more complex, presenting an intriguing area for future research.

Our analysis concentrated on scenarios in which the sender sought to dissuade the receiver from taking actions. Another interesting avenue for future research involves investigating how the analysis and results would vary if the sender's intention was to encourage the receiver to take actions.
\color{black}
\setlength{\bibsep}{2pt}
\bibliographystyle{jpe}
\bibliography{references}

\newpage
\section*{Appendix: Proofs}
\label{sec: appendix}

\renewcommand{\theequation}{A.\arabic{equation}}
\renewcommand{\thefigure}{A.\arabic{figure}}
\setcounter{equation}{0}
\setcounter{figure}{0}

\subsection*{Proof of Lemma~\ref{lemma 1}}
Uniqueness of $p^*$ follows from the fact that both functions are continuous,  $\pi(q)$ is convex, $q$ is linear, $q > \pi(q)$ for $q > 0$ sufficiently small, and  $\pi(1)=2$ (see also Figure~\ref{fig:2}).
For $p \geq q^{ii} $ we have $\pi(q)= 2(q - (1-q)c)$. Thus, when $p^* \geq q^{ii} $, $p^*=\pi(p^*)$ solves to 
\begin{equation}p^*=\frac{2c}{1+2c}.\label{p*-1}\end{equation}
Together with \eqref{equ:q2}, this implies $2\a_G - \a_L \le 1$.
\medskip

When $q^{i}< p < q^{ii}$, we have $\pi(q)= q(1+\a_G) - (1-q)(1+\a_L)c$. Thus, 
when $p^* < q^{ii}$, 
$p^*=\pi(p^*)$ solves to 
\begin{equation}p^*=\frac{(1+\a_L)c}{\a_G + (1+\a_L)c}.\label{p*-2}\end{equation}
Together with ~\eqref{equ:q2}, this implies $2\a_G-\a_L>1$, proving point \ref{lemma:a}. 

Point \ref{lemma:b} of the Lemma (i.e., $p^*>q^m$) is proven by comparing \eqref{p*-1} and \eqref{p*-2} with \eqref{qm}.

Point \ref{lemma:c} of the Lemma (i.e., $q <p^*$ implies $q\geq \pi(q)$) follows from the fact that $\pi(1)>1$. 
\qed

\bigskip

In the proofs of Proposition~\ref{theorem: unconditional}
and Theorems~\ref{theorem:observe} and~\ref{theorem:signal-based},
we will use the following observations.

\begin{enumerate}[label=\arabic*)]
\item
Information revealed in period 1 changes the receiver's belief and serves as a convexification devise. We will therefore study the related problem in which the sender can send information only in period 2, and use revelation in period 1 to convexify the sender's payoff. 
\item
If the receiver acts in period 1, the sender's loss is at least 1, while if the receiver refrains in period 1, the sender's loss is at most 1. Hence, the sender's optimal strategy will attempt to cause the receiver to refrain from acting in period 1.
\end{enumerate}

\medskip


We start by defining a simple class of sender's strategies, called \emph{extreme strategies}, and showing that in some cases it is sufficient to restrict attention to these strategies.

To simplify notation, we denote $u(q)= q-(1-q)c$, and recall that $u(q^m)=0$.
The quantity $u(q)$ is the receiver's expected stage payoff when he acts and his belief is $q$.

\begin{definition}[Extreme strategy]
We define ``extreme'' a sender's strategy that reveals no information in period 1, while in period 2 it:
\begin{itemize}
\item sends the message $g$ with probability $\gamma$ when the state is $G$
(and the message $\ell$ with probability $1-\gamma$);
\item sends the message $\ell$ with probability 1 when the state is $L$.
\end{itemize}
\end{definition}
An extreme strategy splits the belief $q$ of the receiver at the end of period 1 into $\splitted{\frac{(1-\gamma)q}{1-\gamma q},1}$.\footnote{Compare this with our definition of split in the main text in relation to equations (\ref{eq: split}).} 

We state the following:

\begin{lemma}\label{lemma:appendix}
  For each given $q \in [0,1]$, consider the class $\Sigma(q)$ of sender's strategies that (1) reveal no information in period 1, and (2) against which the receiver has an optimal response that refrains in period 1.  For each strategy $\sigma_S\in\Sigma(q)$ there exists an extreme strategy in $\Sigma(q)$ that is at least as good as $\sigma_S$ for the sender.
\end{lemma}

\begin{proof}
Fix an arbitrary strategy $\sigma_S\in\Sigma(q)$ throughout the proof. 
If the receiver refrains in the first period, the sender's message will involve posteriors $(q_j)_{j=1}^n$ with corresponding probabilities $(\b_j)_{j=1}^n$, with $n\ge1$. We necessarily have $\sum_{j=1}^n\b_jq_j=q$ and  $\sum_{j=1}^n\b_j=1$.
  

Let $\sigma_R$ be an optimal receiver's response to $\sigma_S$ such that the receiver refrains in period 1. Because $\sigma_R$ is optimal, any alternative strategy $\sigma'_R$, where the receiver acts in period 1, does not yield to the receiver a higher payoff. 
When the receiver acts in period 1, he observes a signal in addition to the sender's message. 
Each $q_j$ induces two possible posteriors, denoted $q_j^+$ (when the signal was positive) and $q_j^-$ (when the signal was negative). These are given by:
\begin{equation}
\label{equ:64}
q_j^+=\frac{q_i\alpha_G}{q_i\alpha_G + (1-q_i)\alpha_L} \quad\text{and}\quad 
q_j^-=\frac{q_i(1-\alpha_G)}{q_i(1-\alpha_G) + (1-q_i)(1-\alpha_L)}.
\end{equation}

Let $\sigma'_R$ be the strategy that prescribes to act in period 1, and to do so in period 2 only if the receiver's posterior is above $q^m$.  This strategy is optimal for the receiver among all the strategies that prescribe to act in period 1.%
\footnote{Note that the definition of $\sigma'_R$ relies on beliefs rather than signals. This choice is made to streamline the proof and facilitate its presentation.}  Under both strategies $\sigma_R$ and $\sigma'_R$, the receiver acts in period 2 only if the posterior is higher than $q^m$.

The idea of the proof is to iteratively change $\widehat\sigma_S$ by eliminating messages that correspond to various posteriors,
until we are left with two posteriors: 1 and some posterior between $q^i$ and $q^m$.
In Step~1 we consider a given posterior $q_j > q^m$,
and we show that instead of sending the message that corresponds to this posterior,
the sender could have split $q_j$ into $\splitted{0,1}$,
without affecting the property that $\sigma_R$ is better for the receiver than $\sigma'_R$.
In Step~2 we consider a given posterior $q_j < q^i$,
and we show that by properly lowering the probability of reaching the posterior $q_j$, lowering the probability of reaching any posterior $q_k$ (with $q_k > q^i)$, and increasing the probability of reaching the posterior $q_i$,
we still do not affect the property that $\sigma_R$ is better for the receiver than $\sigma'_R$.
Finally, in Step~3 we show that by collecting all posteriors $q_j$ in $[q^i,q^m]$ into their average, we still do not affect the property that $\sigma_R$ is better for the receiver than $\sigma'_R$.

\paragraph{Step 1: Handling posteriors $q_j > q^m$.}

Fix a posterior $q_j > q^m$ that can be attained under $\sigma_S$.  Let $\widehat \sigma_S$ be similar to $\sigma_S$, except that if the message sent to the receiver is the one that leads his belief to $q_j$ (calculated assuming the receiver refrained in period 1), then the sender further splits $q_j$ into $\splitted{0;1}$.

We will show that $\sigma_R$ is weakly better for the receiver than $\sigma'_R$ when facing $\widehat \sigma_S$, so that $\widehat \sigma_S \in \Sigma(q)$.  To this end, denote by $\phi_R(\sigma_S,\sigma_R)$ the receiver's payoff under the strategy pair $(\sigma_S,\sigma_R)$, by $\phi_R(\sigma_S,\sigma_R \mid q_j)$ the receiver's payoff in period 2 conditioned on the sender's message corresponding to the posterior $q_j$, and by $\phi_R(\sigma_S,\sigma_R \mid \lnot q_j)$ the receiver's payoff in period 2 conditioned that the sender's message is not the one corresponding to the posterior $q_j$.  We use analogous notation for the strategies $\widehat\sigma_S$ and $\sigma'_R$.

The definition implies that
\begin{eqnarray}
\label{equ:e1}
\phi_R(\sigma_S,\sigma_R) &=& \beta_j \phi_R(\sigma_S,\sigma_R \mid q_j) + (1-\beta_j)  \phi_R(\sigma_S,\sigma_R \mid \lnot q_j),\\
\label{equ:e2}
\phi_R(\sigma_S,\sigma'_R) &=& u(q) + \beta_j \phi_R(\sigma_S,\sigma'_R \mid q_j) + (1-\beta_j)  \phi_R(\sigma_S,\sigma'_R \mid \lnot q_j),\\
\label{equ:e3}
\phi_R(\widehat \sigma_S,\sigma_R) &=& \beta_j \phi_R(\widehat \sigma_S,\sigma_R \mid q_j) + (1-\beta_j)  \phi_R(\widehat \sigma_S,\sigma_R \mid \lnot q_j),\\
\label{equ:e4}
\phi_R(\widehat \sigma_S,\sigma'_R) &=& u(q) + \beta_j \phi_R(\widehat \sigma_S,\sigma'_R \mid q_j) + (1-\beta_j)  \phi_R(\widehat \sigma_S,\sigma'_R \mid \lnot q_j).
\end{eqnarray}
Since at $q_j$ the strategy $\widehat \sigma_S$ splits $q_j$ into $\splitted{0;1}$,
and in period 2 both $\sigma_R$ and $\sigma'_R$ act at the belief 1 and refrain at the belief 0,
we have
\begin{equation}
\label{equ:e7}
\phi_R(\widehat\sigma_S,\sigma_R \mid q_j) = \phi_R(\widehat\sigma_S,\sigma'_R \mid q_j).
\end{equation}
Since $\sigma_S$ and $\widehat \sigma_S$ coincide in period 1, 
as well as in period 2 at beliefs different than $q_j$,
we have
\begin{eqnarray}
\label{equ:e8}
\phi_R(\widehat\sigma_S,\sigma_R \mid \lnot q_j) &=& \phi_R(\sigma_S,\sigma_R \mid \lnot q_j),\\
\label{equ:e9}
\phi_R(\widehat\sigma_S,\sigma'_R \mid \lnot q_j) &=& \phi_R(\sigma_S,\sigma'_R \mid \lnot q_j).
\end{eqnarray}

We argue that
\begin{equation}
\label{equ:e5}
\phi_R(\sigma_S,\sigma_R \mid q_j) \leq \phi_R(\sigma_S,\sigma'_R \mid q_j).
\end{equation}
Indeed, under $(\sigma_S,\sigma_R)$, at the belief $q_j$, the receiver acts in period 2 (because $q_j > q^m$).  Under $(\sigma_S,\sigma'_R)$, the belief $q_j$ is split to $q_j^+$ and $q_j^-$.  Because $q_j^+ > q_j > q^m$, under $(\sigma_S,\sigma'_R)$ the receiver acts at the belief $q_j^+$.  If under $(\sigma_S,\sigma'_R)$ the receiver acts at $q_j^-$ as well, then \eqref{equ:e5} holds with equality.  Suppose then that under $(\sigma_S,\sigma'_R)$ the receiver refrains at $q_j^-$.  This means that $q_j^- \leq q^m$, so that the receiver's payoff upon acting at $q_j^-$ is non-positive.  The linearity of the payoff function implies that the receiver's payoff upon acting at $q_j$ is a weighted average of his payoff upon acting at $q_j^+$ and his payoff upon acting at $q_j^-$.  This, together with the non-positiveness of the payoff upon acting at $q_j^-$, implies that \eqref{equ:e5} holds (with inequality) in this case.

Since $\sigma_S \in \Sigma(q)$, it follows that 
\begin{equation}
\label{equ:e10}
\phi_R(\sigma_S,\sigma_R) \geq \phi_R(\sigma_S,\sigma'_R).
\end{equation}
This, together with \eqref{equ:e1}, \eqref{equ:e2}, and~\eqref{equ:e5}
implies that 
\begin{equation}
\label{equ:e6}
(1-\beta_j) \phi_R(\sigma_S,\sigma_R \mid \lnot q_j) \geq u(q) + 
(1-\beta_j)  \phi_R(\sigma_S,\sigma'_R \mid \lnot q_j).
\end{equation}
Equations~\eqref{equ:e3}, \eqref{equ:e4}, \eqref{equ:e7}, \eqref{equ:e8}, \eqref{equ:e9}, and \eqref{equ:e6} imply that
\begin{equation}
\phi_R(\widehat\sigma_S,\sigma_R) \geq \phi_R(\widehat\sigma_S,\sigma'_R),
\end{equation}
and therefore $\widehat\sigma_S \in \Sigma(q)$.
Under $(\widehat\sigma_S,\sigma_R)$ the receiver acts in period 2 with lower probability than under $(\sigma_S,\sigma_R)$,
and hence $\widehat\sigma_S$ is better for the sender than $\sigma_S$.


\paragraph{Step 2: Handling posteriors $q_j < q^i$.} 
Let $q_k$ be such that $q_k > q^i$.
Such $q_k$ exists because the weighted average of the posteriors is $q$, which is higher than $q^i$.
Let $\delta$ solve $q^i = \delta q_{k} + (1-\delta) q_{j}$.
Fix $\ep > 0$ such that $\delta \ep \leq \beta_{k}$ and $(1-\delta)\ep \leq \beta_j$.
We define the strategy $\widehat\sigma_S$ to be similar to $\sigma_S$, except that it lowers the probabilities to split to $q_j$ and $q_{k}$, and increases the probability to split to $q^i$.  
That is, $\widehat \sigma_S$ is the sender's strategy that provides no information in period 1 and in period 2 sends 
messages with the following properties:
\begin{itemize}
\item for each $h \in \{1,2,\ldots,n\}\backslash\{j,k\}$, with probability $\beta_{h}$ the receiver's posterior belief in period 2 is $q_{h}$; 
\item with probability $\beta_j - (1-\delta) \ep$ the receiver's posterior belief in period 2 is $q_{j}$;
\item with probability $\beta_{k} - \delta \ep$ the receiver's posterior belief in period 2 is $q_{k}$;
\item with probability $\ep$ the receiver's posterior belief in period 2 is $q^i$.
\end{itemize}
The choice of $\ep$ implies that such a strategy exists.  In words, relative to $\sigma_S$, the strategy $\widehat \sigma_S$ decreases slightly the probability of the posteriors $q_j$ and $q_{h}$, and introduces (or increases the probability of) the posterior $q^i$.  An analogous argument to that in Step~1 shows that $\sigma_R$ is at least as good to the receiver as $\sigma'_R$ when facing $\widehat\sigma_S$.  As a result, $\widehat \sigma_S$ is better for the sender than $\sigma_S$ when $q_{k} > q^m$, and the two strategies perform equally well when $q_{k} \in [q^i,q^m]$.

\paragraph{Step 3: Handling posteriors $q_j \in (q^i,q^m]$.}

Let $\calJ = \bigl\{ j \in \{1,2,\ldots,n\} \colon q^j \in [q^i,q^m]\bigr\}$, let $\overline q$ be the weighted average of all posteriors $(q_j)_{j \in \calJ}$, and let $\overline \beta$ be the sum of the probabilities corresponding to these posteriors:
\begin{equation} 
\overline{q} = \frac{\sum_{j \in \calJ}\beta_j q_j}{\sum_{j \in \calJ}\beta_i}, \qquad \overline \beta = \sum_{j \in \calJ}\beta_i. 
\end{equation}
Consider the sender's strategy $\widehat \sigma_S$ that instead of splitting the receiver's belief to $(q_j)_{j=1}^n$, replaces all beliefs $q_j\le q^m$ by their average $\overline{q}$.  
That is, the strategy $\widehat\sigma_S$ splits $q$ to $(q_j)_{j \not\in \calJ}$ (with probabilities $(\beta_j)_{j \not\in \calJ}$) and $\overline{q}$ (with probability $\overline{\beta}$).

We argue that $\sigma_R$ is at least as good for the receiver as $\sigma'_R$ against $\widehat \sigma_S$.  Indeed, at all $(q_j)_{j\not\in \calJ}$ and at $\overline{q}$, the receiver behaves similarly under $\sigma_R$ (resp.,~under $\sigma'_R$): he refrains in all of them under $\sigma_R$, and he refrains in $q_j^-$ and acts in $q_j^+$ under $\sigma'_R$.  The linearity of the payoff implies that since $\sigma_R$ is at least as good for the receiver as $\sigma'_R$ when facing $\sigma_S$, the same relation holds also when facing $\widehat\sigma_S$.\end{proof}

\subsection*{Proof of Proposition~\ref{theorem: unconditional}}

\paragraph{Step 1: The sender's optimal loss is at most $\N{U}$.}
The sender's payoff in the benchmark case, in which the sender provides no information to the receiver,
is the function $N$ displayed in red in Figure~\ref{fig:1}.
Simple algebraic manipulations show that the line that connects $(q^i,N(q^i))$ and $(1,N(1))$, i.e., the dashed blue line in Figure~\ref{fig:1}, passes below the point $(q^{ii},N(q^{ii}))$.
Therefore, the convexification of $N$, namely, the smallest convex function that is smaller than $N$,
is the function $\N{U}$ defined in the statement of the proposition.
This further implies that the sender's strategy that is described in the statement of the proposition,
denoted $\sigma_S^\U$, ensures that the sender's loss is no more than $\N{U}$.

\bigskip

Since the sender's loss $\N{U}$ is the convexification of $N$, the sender's strategy $\sigma_S^\U$ is optimal among all sender's strategies that do not reveal information in period 2. Our goal is to show that $\sigma_S^\U$ is the sender's optimal strategy also among her strategies that do provide information in period 2. To this purpose, we will show that the payoff guaranteed by any sender's strategy that reveals information only in period 2 is not lower than $\N{U}(q)$. Since information provided in period 1 serves as a convexification device, and since $\N{U}$ is convex, this will imply that $\sigma_S^\U$ is indeed optimal.

We identify different cases which we handle in turn.

\paragraph{Step 2: The case $q \le q^i$.} In this case the receiver's best response to  $\sigma_S^\U$ is RR, hence $\sigma_S^\U$ is indeed optimal for the sender. 
    
\paragraph{Step 3: The case $q > q^m$.} In this case the receiver's payoff in period 1 is positive, hence, when facing a sender's strategy that does not reveal information in period 1, he will act in period 1. Any information that the sender provides in period 2 is provided independently of the receiver's behavior in period 1. Since the receiver acts in period 1, providing this information in period 1 instead of in period 2 cannot harm the sender. However, the optimal strategy that reveals information only in period 1 is $\sigma_S^\U$.

\paragraph{Step 4: The case $q \in (q^i,q^m]$.} 

Fix a sender's strategy that reveals information only in period 2, and suppose that the receiver's best response is to act in period 1. 
As in the case of $q>q^m$, the sender could have revealed the information in period 1 without increasing her loss. Yet the strategy $\sigma_S^\U$ is the optimal sender's strategy that reveals information only in period 1.

Hence, to verify that $\sigma_S^\U$ is optimal for $q \in (q^i,q^m]$, it is sufficient to show that a loss lower than $\N{U}(q)$ cannot be obtained by any sender's strategy that involves information revelation only in period 2 and the receiver refraining in period 1.

From Lemma \ref{lemma:appendix} we know that we can restrict our attention to extreme strategies.
Fix then an extreme strategy of the sender,
and denote by $\gamma$ the probability that the message $g$ is sent when the state is $G$.
Let $\sigma_R$ be the receiver's optimal response,
and assume it refrains in period 1. 
Under $\sigma_R$, the receiver acts in period 2 if and only if his belief is 1, hence the payoff of the receiver and the loss of the sender are both $\gamma q$.
For $\sigma_R$ to be an optimal response, $\gamma q$ must exceed the receiver's payoff from alternative strategies in response to the sender's extreme strategy. These strategies are:
\begin{enumerate}[label=(\alph*)]
\item act in period 1, and act in period 2 when the belief in period 2 is 1.
\item act in period 1, and act in period 2 when the belief in period 2 is 1 or when the signal in period 1 was positive.
\end{enumerate}
Since $q \leq q^m$, the strategy (a) gives the receiver a payoff not higher than $\sigma_R$. As to (b), it gives the receiver
\begin{equation}\label{eq:payoff-strategy-b}
 q-(1-q)c + q\alpha_G - (1-q)\alpha_L c +q(1-\alpha_G)\gamma.
\end{equation}
For the sender, the best extreme strategy is characterised by the minimum $\gamma$ such that $q\gamma$ (the payoff from $\sigma_R$) is not lower than \eqref{eq:payoff-strategy-b}. Denoting the minimum level of $\gamma$ by $\widehat\gamma(q)$, the payoff from the extreme strategy $\sigma_R$ turns out to be:
\begin{equation}
\label{equ:981} q\widehat\gamma(q) = \frac{q-(1-q)c + q\alpha_G - (1-q)\alpha_L c}{\alpha_G},
\end{equation}
and simple algebraic manipulations show that $q \ge q^i$ implies $q\widehat\gamma(q) \ge \N{U}(q)$, proving our result.

\subsection*{Proof of Theorem~\ref{theorem:observe}}

\paragraph{Proof of Part (a).}
As in the proof of Proposition~\ref{theorem: unconditional},
when $q \leq q^i$ and the sender reveals no information,
the receiver's best response is to refrain in both periods,
which is the optimal outcome for the sender.

\paragraph{Proof of Part (b).}
Suppose that $q^i < q \leq p^*$, which implies $q \geq \pi(q)$, and consider the following strategy of the sender,
which is described in the statement of the theorem:
\begin{itemize}
\item in period 1 reveal no information;
\item if the receiver chose A in period 1, then in period 2 reveal no information;
\item if the receiver chose R in period 1, then in period 2 split $q$ into  $\splitted{\xi(q);1}$.\footnote{In other words, if the state of nature is $G$, reveal it with probability $\pi(q)$, which is less than 1 for $q\leq p^*$, and keep silent otherwise.} 
\end{itemize}
Observe that $\xi(q)$ is decreasing in $q$. 
This is proved by considering that, with $q^i<q<q'$, we have
\begin{equation*}
  \xi(q')<\xi(q) \iff \frac{1-\pi(q)}{1-q}<\frac{\pi(q')-\pi(q)}{q'-q};
\end{equation*}
the latter inequality follows from the fact that the LHS is strictly lower -- while the RHS is always equal or higher -- than $1+\a_G+(1+\a_L)c$, which is the slope of $\pi(q)$ in the interval $(q^i,q^{ii})$.

Since $\pi(q^i)=0$ implies $\xi(q^i)=q^i$, and since $\pi(p^*)=p^*$, for any $q\in(q^i,p^*)$ we have in particular $0<\xi(q)<q^i$.
\medskip

When the receiver responds with the compliance strategy, he plays A only in period 2, and this occurs with probability $\pi(q)$. The costs to the sender is therefore $\pi(q)$, and the receiver's payoff is $\pi(q)$.

We verify that this pair of strategies is a Stackelberg equilibrium.
Since $\pi(q)$ is the 
best possible outcome for the sender, 
she cannot profit by deviating.

If the receiver deviates and plays A in period 1, he will get no information from the sender in period 2, and his optimal payoff would still be $\pi(q)$. Hence the receiver cannot profit by deviating either.
On the other hand, in period 2, when the receiver's belief is $\xi(q)$ his best response is R, because  $\xi(q)<q^i < q^m$.

\paragraph{Proof of Part (c).}
Assume $q > p^*$.
We distinguish between two cases. 

\paragraph{Case 1: $p^* > q^{ii}$.} 
Since $q > q^{ii}$, $\pi(q)$ is linear on the interval $[p^*,1]$.
Consider the following sender's strategy:
\begin{itemize}
\item in period 1, splits  $q$ into $\splitted{p^*;1}$; 
\item in period 2, follow the strategy as described in the proof of Part (b).
Namely, if the sender chose A in period 1, then provide no information in period 2.
If the sender chose R in period 1, and the split realization is $p^*$,
further split it in period 2 into $\langle 0;1\rangle$. Otherwise do nothing. 
\end{itemize}

As we have seen, the receiver's best response is as follows:
if the split realization in period 1 is to 1, follow strategy AA;
if the split realization in period 1 is to $p^*$, choose R in period 1,
and in period 2 choose A if and only if the split realization in period 2 is to 1.

If the split realization in period 1 is $p^*$,
the receiver's payoff is 
\begin{equation}
\left(1-\frac{q-p^*}{1-p^*}\right)\pi(p^*)+\frac{q-p^*}{1-p^*}\pi(1)=\pi(q),
\end{equation}
where the equality is due to the
linearity of $\pi$ on this interval.
Hence, as in Part~(b), the receiver cannot profit by deviating.
The sender's expected loss under this strategy is, therefore, $\pi(q)$, and hence, as in Part~(b), the sender cannot profit by deviating as well. 

\paragraph{Case 2: $p^*\le q^{ii}$.} Consider the strategy pair introduced in Case 1. 
Since $\pi(q)$ is convex on $[p^*,1]$, the receiver's strategy is the best response to the sender's strategy. 
However, 
$\pi(q)$ is no longer linear on $[p^*, 1]$, and therefore the sender does not attain the lower bound $\pi(q)$. Hence, it is not clear whether the sender's strategy is optimal.

When the receiver's belief in period 2 is $q$,
his highest payoff is obtained when the sender reveals the state,
and he, the receiver, chooses A when the state is $G$, 
and R when it is $L$.
This results in a payoff $q$ to the receiver.

Since $q > p^*$, 
we have $\pi(q) > q$:
even ignoring information revealed by the sender,
the receiver can guarantee more than $q$.
This implies that in any equilibrium,
the sender chooses A with positive probability in period 1.



If the receiver chooses A in period 1, and in period 2 he chooses A only if the signal was positive, 
then the cost to the sender is:
\begin{equation}\label{eq:cost-to-sender-1}
\begin{dcases*}
    1+ q\a_G+(1-q)\a_L, &  if $p^* \le q <q^{ii}$, \\
     2, & if $q^{ii} \le q$.
\end{dcases*}
\end{equation}


Cost minimization for the sender
involves the convexification of the function that coincides with $\pi(q)$ when $q\leq p^*$ and with \eqref{eq:cost-to-sender-1} when $q > p^*$. For $q>p^*$, this function is linear, equals $\pi(q)$ at $q=p^*$ and equals 2 at $q=1$. 
For $q<1$ this linear function lies strictly above $\pi(q)$, which implies that the optimal outcome for the sender is worse than the benchmark (see Figure~\ref{fig:1}).

The optimal messaging strategy that corresponds to this convexification is as follows:
\begin{itemize}
\item in period 1, split $q$ into $\splitted{p^*;1}$;
\item if the receiver chose A in period 1, reveal no information in period 2;
\item if the receiver chose R in period 1 and the split realization was to $p^*$,
in period 2 split the receiver's belief into $\splitted{0;1}$. 
\end{itemize}
By adopting this strategy, the sender provides some information about $q$ in period 1, and fully reveals the state in period 2. 
The proof of Theorem \ref{theorem:observe} is complete.

\subsection*{Proof of Theorem~\ref{theorem:signal-based}}

\paragraph{Proof of Part (a).}
Consider the following strategy of the sender:
\begin{itemize}
\item in period 1 reveal no information;
\item if the signal in period 1 is negative, reveal no information in period 2;
\item if the signal in period 1 is positive and if the state is $G$, then in period 2 with probability $\gamma$ send the message $g$, otherwise, send the message $\ell$; that is to say, split $q$ into $\splitted{\frac{(1-\gamma)q}{1-\gamma q};1}$. 
\end{itemize}

If the receiver complies (i.e., he refrains in period 1 and acts in period 2 if and only if the message is $g$), his payoff is $q\gamma$. Hence, a necessary condition for compliance to be a best response is that the resulting payoff is no less than the payoff from ignoring the sender's message:
\begin{equation}\label{eq:2-condition}
  q\gamma\ge \pi(q). 
\end{equation}
Since $\gamma \in [0,1]$,
\eqref{eq:2-condition} can hold only when $q \geq \pi(q)$,
that is, $q \leq p^*$.
As we showed in the proof of Theorem~\ref{theorem:observe},
when $q \leq p^*$ we have $\xi(q)<q^i$.
The condition~\eqref{eq:2-condition} can be equivalently expressed as $\frac{(1-\gamma)q}{1-\gamma q}\le\xi(q)$. 
It follows that, when \eqref{eq:2-condition} holds, the posterior belief in period 2 induced by the message $\ell$ is also lower than $q^i$ and, therefore, lower than $q^m$.

Consider now the alternative strategy for the receiver where he acts in period 1,
and in period 2 he acts if and only if the signal is positive and the message is $g$.
The receiver's payoff in this case is $u(q) + q\a_G\gamma
$.

Therefore, a second condition that is necessary for compliance to be a best response is $q\gamma\ge u(q) + q\a_G\gamma$, or:
\begin{equation}\label{eq:1-condition}
  q\gamma\ge \frac{u(q)}{1-\a_G}.
\end{equation}

Condition~\eqref{eq:1-condition} can be compared to \eqref{eq:2-condition} to establish which of the two is binding at each $q$. 
Namely, condition~\eqref{eq:2-condition} will be binding when
\begin{equation}\label{eq:which-binding}
\frac{u(q)}{1-\a_G} \le \pi(q),
\end{equation}
with $\gamma\le1$ implying that neither side can exceed $q$.

Recalling that $\pi(q)$ is the maximum of three functions, with one of them being $2u(q)$, we conclude that the inequality \eqref{eq:which-binding} holds at all $q$ if $u(q)/(1-\a_G)\le 2u(q)$, that is, if $\a_G\le 1/2$.
This is enough to prove part (a) of the Theorem. Indeed, if \eqref{eq:which-binding} is always satisfied and \eqref{eq:2-condition} is the only condition for compliance to be a best receiver's response, 
then there is no difference between this case and the case of action-based information discussed in Theorem~\ref{theorem:observe}.

\paragraph{Proof of Part (b).}
When $\a_G>1/2$, the identity of the condition among \eqref{eq:2-condition} and \eqref{eq:1-condition} that is binding depends on $q$ and $\gamma$. In this case, the optimal choice of $\gamma$ for the sender corresponds to the minimal $\gamma\in[0,1]$ that satisfies both inequalities~\eqref{eq:2-condition} and~\eqref{eq:1-condition}.

Before characterizing the optimal strategy, consider that, since $\gamma\leq1$, condition \eqref{eq:2-condition} cannot be satisfied when $q<\pi(q)$, i.e., when $q>p^*$.
Similarly, conditions \eqref{eq:1-condition} cannot be satisfied for $\gamma\le 1$ when $q<u(q)+\a_Gq$, i.e., when $q$ exceeds the level
\begin{equation}\label{eq:p**}
p^{**}:= \frac{c}{\a_G+c}
\ge \frac{c}{1+c} = q^m.
\end{equation}
Moreover, when $\a_G>1/2$, we have $u(q)+\a_Gq>\pi(q)$. This follows from the fact that for $q\in(q^i,q^{ii})$ 
we have
\begin{equation*}\begin{split}
  u(q)+\a_Gq-\pi(q) &= q-(1-q)c+\a_Gq - [q(1+\a_G)-(1-q)(1+\a_L)c]\\
 &=\a_L(1-q)c>0,\end{split}
\end{equation*}
while for $q\ge q^{ii}$ we have
\begin{equation*}\begin{split}
  u(q)+\a_Gq-\pi(q) &= (1+\a_G)q-(1-q)c - 2[q-(1-q)c]\\
  &=(1-q)c-(1-\a_G)q>0.
\end{split}
\end{equation*}
Therefore, 
\begin{equation}
\label{equ:65}
p^{**}=u(p^{**})+\a_Gp^{**}>\pi(p^{**}), 
\end{equation}
which implies $p^{**}<p^*$.

\medskip

We distinguish between three cases.

\paragraph{Case 1: $q \le p^{**}$.}

Because in condition \eqref{eq:which-binding} the LHS is linear, the RHS is convex, the LHS is lower than the RHS at  $q=0$, while the RHS is lower than the LHS at $q=p^{**}$ (see \eqref{equ:65}), 
there is $p^e \in (0,p^{**})$ such that for $q < p^e$ condition~\eqref{eq:2-condition} is the binding constraint,  while for $q> p^e$ condition~\eqref{eq:1-condition} is the binding constraint.
Since at $q^m$ the LHS of \eqref{eq:which-binding} is zero while the RHS is positive, we have $p^e>q^m$.  
Since $u(q^{ii})/(1-\a_G)> 2u(q^{ii}) = \pi(q^{ii})$, we have $p^e<q^{ii}$.

By providing information only in period 2, the best the sender can do is as
follows.
\begin{itemize}
\item When $q\le p^e$, condition~\eqref{eq:2-condition} is the binding constraint.  In this case the receiver's payoff coincides with the action-based information payoff, which is $\pi(q)$, and the sender's cost is therefore $\pi(q)$.
\item When $p^e <q \le p^{**}$, the sender can still prevent the receiver from acting in period 1, but then the receiver's payoff coincides with the RHS of condition~\eqref{eq:1-condition}, which is $u(q)/(1-\a_G)$, and the sender's cost is accordingly $u(q)/(1-\a_G)$.
\end{itemize}

\paragraph{Case 2: $p^{**} < q \le q^{ii }$.} When $p^{**} < q^{ii}$, no amount of information revealed in period 2 makes action R better than action A in period 1, hence the receiver will choose A in period 1.
Since the sender observes the signal obtained by the receiver,
she knows the receiver's belief in period 2 (calculated under the assumption that the receiver acted in period 1).
Since $q \leq q^{ii}$,
if the signal in period 1 is negative, the receiver's belief in period 2 is at most $q^m$, and the best response for the sender in period 2 is to reveal no information.
If the signal in period 1 is positive, the receiver's belief in period 2 is above $q^m$, and the best response for the sender in period 2 is to split the receiver's belief between $\splitted{q^m;1}$.
As a result, the sender's optimal cost is 
\begin{equation}
\label{equ:66}
1 + \big(q\a_G+(1-q)\a_L\big)\times \left(\frac{q^{+}-q^m}{1-q^m}\right),
\end{equation}
where, as in \eqref{equ:64},
\begin{equation}\label{eq:q+}
q^{+}= \frac{q\a_G }{q\a_G +(1-q)\a_L}    
\end{equation}
is the receiver's posterior belief after a positive signal. The expression in \eqref{equ:66} simplifies to
\begin{equation}\label{eq:simplified-payoff-case-2.3}
    1+q \alpha _G- (1-q) \alpha _L c.
\end{equation}
which is larger than $\pi(q)$ on $(p^{**},q^{ii}]$.


\paragraph{Case 3: $q>\max\{p^{**},q^{ii}\}$.}
In this case, the receiver chooses A in period 1, and his belief in period 2 is above $q^m$, whatever signal he obtains from his action in period 1.
Hence the best response for the sender in period 2 is to split the receiver's belief into $\splitted{q^m;1}$,
whether the signal in period 1 is positive or negative. Doing so, leads to sender's cost of $1 + \frac{q-q^m}{1-q^m}$,
which is larger than $\pi(q)$ for $q>\max\{p^{**},q^{ii}\}$.


To summarize, 
when $\a_G > 1/2$, 
the sender's cost when information can be provided only in period 2 is given by (see Figure~\ref{fig:cases-p**}):
\begin{equation}
\label{equ:C}
\begin{dcases*}
   \pi(q), & if $q\le p^e$, \\
 \dfrac{u(q)}{1-\a_G}, & if $p^e <q \le p^{**}$,\\
    1+q \alpha _G- (1-q) \alpha _L c,  & if $p^{**}< q \le q^{ii }$,\\
   1+\frac{q-q^m}{1-q^m}, & if $\max\{p^{**},q^{ii}\} <q \le 1$.
  \end{dcases*}
\end{equation}
Note that, when $p^{**} > q^{ii}$, the third subcase in the definition of \eqref{equ:C} is empty.

\begin{figure}
\makebox[\textwidth][c]{
\begin{minipage}{.6\textwidth}
\centering
\includegraphics[width=\linewidth]{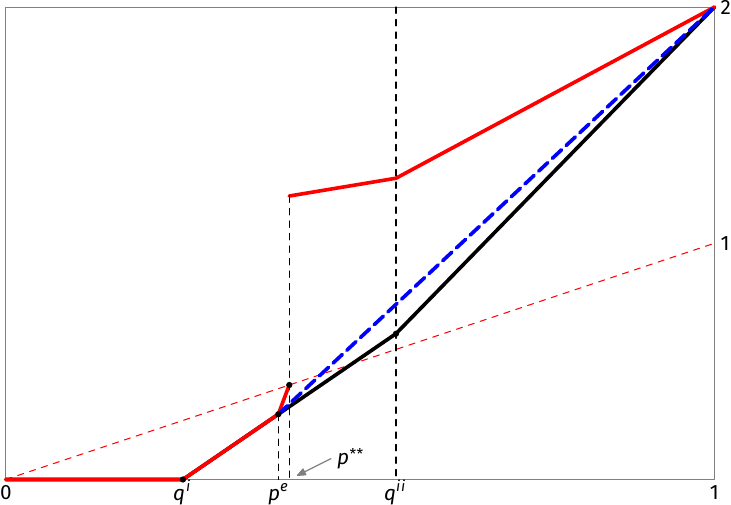}
(a) The case $p^{**}<q^{ii}$
\end{minipage}
\hspace{2mm}
\begin{minipage}{.6\textwidth}
\centering
\includegraphics[width=\linewidth]{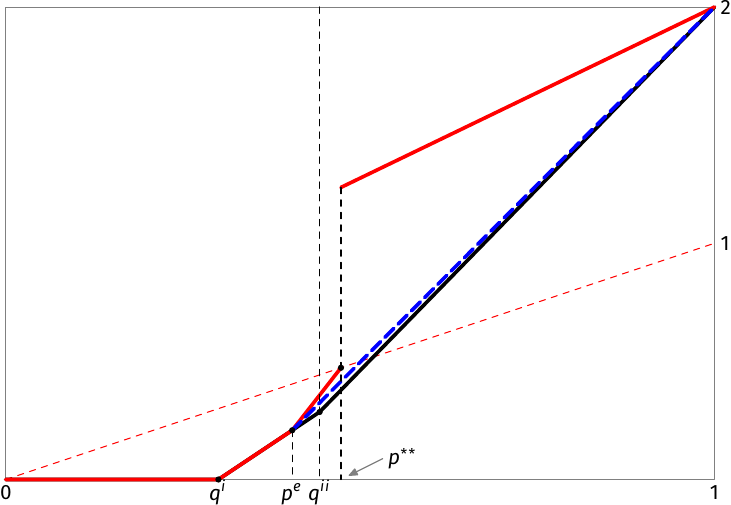}
(b) The case $p^{**}>q^{ii}$
\end{minipage}}
\caption{Sender's payoff in the proof of Theorem~\ref{theorem:signal-based}, with kink points at  $q^{i}$ and $q^{ii}$.}
\label{fig:cases-p**}
\end{figure}

Since the sender wishes to minimize her cost, her Stackelberg equilibrium payoff is  the convexification of the function defined in \eqref{equ:C}.
Tedious computations show that the convexification is linear between $p^e$ and 1,
and has a kink point at $p^e$. 
Therefore, when $q>p^e$ the optimal messaging strategy is as follows:
\begin{itemize}
\item in period 1, split $q$ into $\splitted{p^e;1}$;
\item if the split realization is $p^e$, further split $p^e$ into
  $\splitted{\xi(p^e);1}$, and do nothing otherwise.
\end{itemize}
Notice that this optimal strategy renders $p^{**}$ irrelevant.

\paragraph{General sender's strategies}

So far we have assumed that the  sender can use only extreme strategies.
Since in all Stackelberg equilibria,
the receiver refrains whenever her belief is not 1,
the same argument as in the proof of Proposition~\ref{theorem: unconditional} shows that all sender's strategies are weakly dominated by extreme strategies.
This completes the proof. \qed

\end{document}